\documentclass[a4paper,10pt]{article}
\usepackage{lmodern}
\usepackage[utf8]{inputenc}
\usepackage[T1]{fontenc}
\usepackage{geometry}
\usepackage{microtype}
\usepackage{authblk}
\usepackage{amsmath,amssymb}
\usepackage{amsthm}
\usepackage{mathtools}
\usepackage{graphicx}
\usepackage{xcolor}
\usepackage{xspace}
\usepackage{booktabs}
\usepackage{subcaption}
\usepackage{hyperref}
\hypersetup{
    colorlinks,
    linkcolor={red!50!black},
    citecolor={blue!50!black},
    urlcolor={blue!80!black}
}

\theoremstyle{definition}
\newtheorem{thm}{Theorem}[section]
\newtheorem{defn}[thm]{Definition}
\newtheorem{lem}[thm]{Lemma}
\newtheorem{rem}[thm]{Remark}

\newtheorem*{commr}{\textcolor{red}{Comment}}
\newtheorem*{commb}{\textcolor{blue}{To do}}

\newcommand{\mbfb}{\mathbf b}
\newcommand{\mbfc}{\mathbf c}
\newcommand{\mbfe}{\mathbf e}

\newcommand{\mbfy}{\mathbf y}
\newcommand{\mbfA}{\mathbf A}

\newcommand{\mbfD}{\mathbf D}

\newcommand{\mbfM}{\mathbf M}
\newcommand{\mbfP}{\mathbf P}

\renewcommand{\O}{\mathcal O}
\newcommand{\R}{\mathbb R}

\newcommand{\abs}[1]{\lvert #1\rvert}

\DeclareMathOperator*{\eq}{=}

\title{On order conditions for modified Patankar-Runge-Kutta schemes}
\author[a]{Stefan Kopecz}
\author[a]{Andreas Meister}
\affil[a]{Institute of Mathematics, University of Kassel}
\date{\today}

\usepackage{mdframed}
\newmdenv[
  linecolor=blue,
  linewidth=2pt,
  topline=false,
  bottomline=false,
  skipabove=\topsep,
  skipbelow=\topsep,
  leftmargin=-10pt,
  rightmargin=-10pt,
  innertopmargin=0pt,
  innerbottommargin=0pt
]{siderules}

\DeclareFontFamily{U}{mathx}{\hyphenchar\font45}
\DeclareFontShape{U}{mathx}{m}{n}{<-> mathx10}{}
\DeclareSymbolFont{mathx}{U}{mathx}{m}{n}
\DeclareMathAccent{\widebar}{0}{mathx}{"73}
\newcommand{\PPDS}{\widehat P}
\newcommand{\DPDS}{\widehat D}
\newcommand{\pPDS}{\widehat p}
\newcommand{\dPDS}{\widehat d}

\begin{document}
\maketitle
\abstract{In \cite{BDM2003} the modified Patankar-Euler and modified Patankar-Runge-Kutta schemes were introduced to solve positive and conservative systems of ordinary differential equations. 
These modifications of the forward Euler scheme and Heun's method guarantee positivity and conservation irrespective of the chosen time step size.
In this paper we introduce a general definition of modified Patankar-Runge-Kutta schemes and derive necessary and sufficient conditions to obtain first and second order methods. We also introduce two novel families of second order modified Patankar-Runge-Kutta schemes.}
\section{Introduction}
We consider production-destruction systems (PDS) of the form
\begin{align}\label{eq:pds}
 \frac{d y_i}{dt}(t)=P_i(\mbfy(t))-D_i(\mbfy(t)),\quad i=1,\dots,N.
\end{align}
By $\mbfy=(y_1,\dotsc, y_N)^T$ we denote the vector of constituents, which depends on time $t$. Both, the production terms $P_i$ and the destruction terms $D_i$ are assumed to be non-negative, that is $P_i, D_i\geq 0$ for $i=1,\dots,N$.
Furthermore, the production and destruction terms can be written as
\begin{align}\label{eq:pijdij}
 P_i(\mbfy) = \sum_{j=1}^N p_{ij}(\mbfy),\quad D_i(\mbfy) = \sum_{j=1}^N d_{ij}(\mbfy),
\end{align}
where $d_{ij}(\mbfy)\geq 0$ is the rate at which the $i$th constituent transforms into the $j$th component, while $p_{ij}(\mbfy)\geq 0$ is the rate at which the $j$th constituent transforms into the $i$th component.

We are interested in PDS which are positive as well as fully conservative.
\begin{defn}
The PDS \eqref{eq:pds} is called \textit{positive}, if positive initial values,  $y_i(0)> 0$ for $i=1,\dots,N$, imply positive solutions, $y_i(t)>0$ for $i=1,\dots,N$, for all times $t>0$.
\end{defn}
\begin{defn}\label{defn:consode}
The PDS \eqref{eq:pds}, \eqref{eq:pijdij} is called \textit{conservative}, if for all $i,j=1,\dots,N$ and $\mbfy\geq 0$, we have
$
p_{ij}(\mbfy) = d_{ji}(\mbfy). 
$
The system is called \textit{fully conservative}, if in addition
$
p_{ii}(\mbfy)=d_{ii}(\mbfy)=0
$
holds for all $\mbfy\geq 0$ and $i=1,\dots,N$.
\end{defn}
In the following, we will assume that the PDS \eqref{eq:pds} is fully conservative.
Remark~\ref{rem:fullycons} shows that every conservative PDS can be rewritten as an equivalent fully conservative PDS. 
\begin{rem}\label{rem:fullycons}
If $p_{ii}=d_{ii}\ne0$ for some $i\in\{1,\dots,N\}$ in \eqref{eq:pijdij}, we can write 
\[
P_i(\mbfy)-D_i(\mbfy) = \sum_{\substack{j=1\\j\ne i }}^N\bigl( p_{ij}(\mbfy)-d_{ij}(\mbfy)\bigr) + \underbrace{p_{ii}(\mbfy)-d_{ii}(\mbfy)}_{=0} =  \sum_{\substack{j=1\\j\ne i }}^N \bigl(p_{ij}(\mbfy)-d_{ij}(\mbfy)\bigr).
\]
Setting $\widetilde p_{ij}=p_{ij}$, $\widetilde d_{ij}=d_{ij}$ for $i\ne j$ and $\widetilde p_{ii}=\widetilde d_{ii}=0$, results in
\[
P_i(\mbfy)-D_i(\mbfy)=\sum_{j=1}^N \left(\widetilde p_{ij}(\mbfy) - \widetilde d_{ij}(\mbfy)\right).
\]
Thus, we have found an equivalent fully conservative PDS.
\end{rem}
\begin{rem}
In case of a fully conservative PDS, \eqref{eq:pijdij} can be written as
\[
P_i(\mbfy) = \sum_{\substack{j=1\\j\ne i }}^N p_{ij}(\mbfy),\quad D_i(\mbfy) = \sum_{\substack{j=1\\j\ne i }}^N d_{ij}(\mbfy).
\]
But for the sake of a simple notation, we will always use the form \eqref{eq:pijdij}.
\end{rem}
Examples of positive and conservative PDS, which model academic as well as realistic applications, can be found in Section~\ref{sec:testcases}.

If a PDS is conservative the sum of its constituents $\sum_{i=1}^N y_i(t)$ remains constant in time, since we have
\[
\frac{d}{dt}\sum_{i=1}^N y_i=\sum_{i=1}^N \bigl(P_i(\mbfy)-D_i(\mbfy)\bigr) = \sum_{i,j=1}^N \bigl(p_{ij}(\mbfy)-d_{ij}(\mbfy)\bigr) = \sum_{i,j=1}^N \bigl(\underbrace{p_{ij}(\mbfy)-d_{ji}(\mbfy)}_{=0}\bigr) = 0.
\]
This motivates the definition of a conservative numerical scheme.
\begin{defn}
Let $\mbfy^n$ denote an approximation of $\mbfy(t^n)$ at time level $t^n$. 
 The one-step method
 \[
 \mbfy^{n+1} = \mbfy^n + \Delta t \Phi(t^n,\mbfy^n,\mbfy^{n+1},\Delta t)
 \]
 is called
 \begin{itemize}
 \item \textit{unconditionally conservative}, if 
 \[
 \sum_{i=1}^N \left(y_i^{n+1}-y_i^n\right)=0
 \]
 is satisfied for all $n\in\mathbb N$ and $\Delta t>0$.
 \item \textit{unconditionally positive}, if it guarantees $\mbfy^{n+1}>0$ for all $\Delta t>0$ and $\mbfy^n>0$.
 \end{itemize}

\end{defn}

The modified Patankar-Euler and modified Patankar Runge-Kutta scheme were introduced in \cite{BDM2003} to guarantee unconditional conservation and positivity of the numerical solution of a conservative and positive PDS.
Both schemes are members of the more general class of modified Patankar-Runge-Kutta (MPRK) schemes as defined in Definition~\ref{def:MPRKdefn} below.
The modified Patankar-Euler scheme reads
\begin{align*}
y_i^{n+1} = y_i^n + \Delta t\sum_{j=1}^N\left( p_{ij}(\mbfy^n)\frac{y_j^{n+1}}{y_j^n} - d_{ij}(\mbfy^n)\frac{y_i^{n+1}}{y_i^n}\right),\quad i=1,\dots,N,
\end{align*}
and is unconditionally positive, conservative and first order accurate.
It can be understood as a modification of the forward Euler method, in which the production and destruction terms are weighted in a way to ensure unconditional positivity and conservation of the numerical solution.
We see that the explicitness of the forward Euler scheme is lost and the solution of a linear system of size $N \times N$ is required to obtain the approximation at the next time level.
It is noteworthy that even when the PDS is nonlinear, only a linear system has to be solved.
The second order modified Patankar-Runge-Kutta scheme is given by
\begin{align*}
y_i^{(2)} &= y_i^n + \Delta t\sum_{j=1}^N\left( p_{ij}(\mbfy^n)\frac{y_j^{(2)}}{y_j^n} - d_{ij}(\mbfy^n)\frac{y_i^{(2)}}{y_i^n}\right),\\
y_i^{n+1} &= y_i^n + \frac{\Delta t}2 \sum_{j=1}^N\left( \left(p_{ij}(\mbfy^n)+p_{ij}(\mbfy^{(2)})\right)\frac{y_j^{n+1}}{y_j^{(2)}} - \left(d_{ij}(\mbfy^n)+d_{ij}(\mbfy^{(2)})\right)\frac{y_i^{n+1}}{y_i^{(2)}}\right),
\end{align*}
for $i=1,\dots,N$. This is an unconditionally positive and conservative modification of Heun's predictor corrector method,
which requires the solution of two linear systems of size $N\times N$ in each time step.

Both schemes have been successfully applied to solve physical, biogeochemical and ecosystem models (\cite{BDM2005,BBKMNU2006,BMZ2009,HenseBurchard2010,HenseBeckmann2010,MeisterBenz2010,WHK2013}).
They have also proven beneficial in cosmology \cite{KlarMuecket2010}. 
In \cite{SchippmannBurchard2011} it was demonstrated that the second order scheme of \cite{BDM2003} outperforms standard Runge-Kutta and Rosenbrock methods when solving biogeochemical models without multiple
source compounds per system reactions. 
The same was shown with respect to workload in \cite{BonaventuraDellaRocca2016}, where the Brusselator PDS was solved with different time integration schemes.

In \cite{BBKS2007,BRBM2008} second order schemes, which ensure conservation in a biochemical sense, were introduced. 
These schemes require the solution of a non-linear equation in each time step. 
Other schemes for the same purpose were recently presented in \cite{RadtkeBurchard2015}. 
These explicit schemes incorporate the MPRK schemes of \cite{BDM2003} to achieve multi-element conservation for stiff problems.
A potentially third order Patankar-type scheme was introduced in \cite{FormaggiaScotti2011}.
This scheme uses the MPRK scheme of \cite{BDM2003} as a predictor and applies a corrector which is based on a BDF method.

Modified Patankar-Runge-Kutta type schemes are also used in the context of partial differential equations. An implicit first order Patankar-type scheme based on a third order SDIRK method was presented in \cite{MeisterOrtleb2014} and applied to the shallow water equations.

In the present paper we will generalize the results of \cite{BDM2003} and introduce a more general class of unconditionally positive and conservative schemes based on explicit Runge-Kutta schemes.
In particular, we want to avoid the solution of non-linear equations and to keep the linear implicity of the methods of \cite{BDM2003}.
Furthermore, we are interested in conservation as defined in Definition~\ref{defn:consode}, biochemical conservation is not of interest in this paper.


Until now, a general introduction and investigation of modified Patankar-Runge-Kutta schemes is lacking.
This is the purpose of the present paper. 
In particular, we present necessary and sufficient conditions to obtain first and second order accurate schemes. These show that the Patankar-weights chosen in \cite{BDM2003} are not the only possible choices and are not applicable to general Runge-Kutta schemes.
 
The paper is organized as follows. In Section~\ref{sec:MPRK} a general definition of modified Patankar-Runge-Kutta (MPRK) schemes will be given. It will be shown that MPRK schemes are unconditionally positive and conservative by construction. Sections~\ref{sec:MPRKorder1} and \ref{sec:MPRKorder2} deal with the construction of MPRK schemes of first and second order. We present necessary and sufficient conditions to obtain a certain order along with novel MPRK schemes. Finally, the test problems of Section~\ref{sec:testcases} are used in Section~\ref{sec:numres} to compare the new MPRK schemes with the schemes introduced in \cite{BDM2003}.
\section{Modified Patankar-Runge-Kutta schemes}\label{sec:MPRK}
An explicit $s$-stage Runge-Kutta method for the solution of an ordinary differential equation
$y'(t) = f(t,y(t))$
is given by
\begin{equation*}
\begin{split}
 y^{(k)} &= y^n + \Delta t\sum_{\nu=1}^{k-1} a_{k\nu} f(t^n+c_\nu \Delta t,y^{(\nu)}),\quad k=1,\dots,s,\\
 y^{n+1}&=y^n+\Delta t \sum_{k=1}^s b_k 
 f(t^n+c_k\Delta t,y^{(k)}).
\end{split}
\end{equation*}
The method is characterized by its coefficients $a_{k\nu}$, $b_k$, $c_k$ for $k=1,\dots,s$, $\nu=1,\dots,k-1$ and can be represented 
by the Butcher tableau
\[
\begin{array}{c|c}
\mbfc        &    \mbfA\\\hline
          & \mbfb
\end{array},
\]
with $\mbfA = (a_{k\nu})_{k,\nu=1,\dots,s}$, $\mbfc=(c_1,\dots,c_s)^T$ and $\mbfb=(b_1,\dots,b_s)$.
Applied to \eqref{eq:pds} the method reads
\begin{subequations}
\begin{align}\label{rkscheme}
 y_i^{(k)} &= y_i^n + \Delta t\sum_{\nu=1}^{k-1} a_{k\nu} 
 \sum_{j=1}^N \left(p_{ij}(\mbfy^{(k)})-d_{ij}(\mbfy^{(k)})\right),\quad k=1,\dots,s,\\ \label{eq:rkschemenp1}
 y_i^{n+1}&=y_i^n+\Delta t \sum_{k=1}^s b_k 
 \sum_{j=1}^N\left( p_{ij}(\mbfy^{(k)})-d_{ij}(\mbfy^{(k)})\right).
\end{align}
\end{subequations}

The idea of the modified Patankar-Runge-Kutta schemes is to adapt explicit Runge-Kutta schemes in such a way that they become positive irrespective of the chosen time step size $\Delta t$, while still maintaining their inherent property to be conservative. 
One approach to achieve unconditional positivity is the so-called Patankar-trick introduced in \cite{Patankar1980} as \textit{source term linearization} in the context of turbulent flow. If we modify \eqref{eq:rkschemenp1} and add a weighting of the destruction terms like
\begin{equation*}
y_i^{n+1} = y_i^n + \Delta t\sum_{k=1}^{s} b_{k} 
 \sum_{j=1}^N\biggl( p_{ij}(\mbfy^{(k)})-d_{ij}(\mbfy^{(k)})\frac{y_i^{n+1}}{\sigma_i}\biggr),
\end{equation*} 
we obtain
\begin{equation*}
y_i^{n+1} = \frac{y_i^n + \Delta t\sum_{k=1}^s b_k 
 \sum_{j=1}^N p_{ij}(\mbfy^{(k)})}{1 + \Delta t\sum_{k=1}^s b_k 
 \sum_{j=1}^N d_{ij}(\mbfy^{(k)})/\sigma_i}.
\end{equation*} 
Thus, if $y_i^n$, the weights $b_k$ for $k=1,\dots,s$ and $\sigma_i$ are positive, so is $y_i^{n+1}$. 
The crucial idea of the Patankar-trick is to multiply the destruction terms with weights that comprise $y_i^{n+1}$ as a factor themselves.

Weighting only the destruction terms will result in a non-conservative scheme. So the production terms have to be weighted accordingly as well. Since we have $d_{ij}(\mbfy)=p_{ji}(\mbfy)$, the proper weight for $p_{ij}(\mbfy^{(k)})$ is $y_j^{n+1}/\sigma_j$.

The above ideas lead to the following definition.
\begin{defn}\label{def:MPRKdefn}
Given a non-negative Runge-Kutta matrix $\mbfA=(a_{ij})_{i,j=1,\dots,s}$, non-negative weights $b_1,\dotsc,b_s$ and $\delta\in\{0,1\}$, the scheme
\begin{subequations}\label{eq:MPRK}
\begin{align}\label{eq:MPRKstages}
y_i^{(k)} &= y_i^n + \Delta t\sum_{\nu=1}^{k-1} a_{k\nu} 
 \sum_{j=1}^N \biggl(p_{ij}(\mbfy^{(\nu)})(1-\delta)+p_{ij}(\mbfy^{(\nu)})\frac{y_j^{(k)}}{\pi^{(k)}_j}\delta-d_{ij}(\mbfy^{(\nu)})\frac{y_i^{(k)}}{\pi^{(k)}_i}\biggr),\quad k=1,\dots,s,\\\label{eq:MPRKapprox}
 y_i^{n+1}&=y_i^n+\Delta t \sum_{k=1}^s b_k 
 \sum_{j=1}^N\biggl( p_{ij}(\mbfy^{(k)})\frac{y_j^{n+1}}{\sigma_j}-d_{ij}(\mbfy^{(k)})\frac{y_i^{n+1}}{\sigma_i}\biggr),
\end{align}
\end{subequations}
for $i=1\dots,N$, is called \textit{modified Patankar-Runge-Kutta scheme} (MPRK) if
\begin{enumerate}
 \item $\pi_i^{(k)}$ and $\sigma_i$ are unconditionally positive for $k=1,\dots, s$ and $i=1,\dots,N$, 
 \item $\pi_i^{(k)}$ is independent of $y_i^{(k)}$ and $\sigma_i$ is independent of $y_i^{n+1}$ for $k=1,\dots, s$ and $i=1,\dots,N$.
\end{enumerate}
The weights $1/\sigma_i$ and $1/\pi_i^{(k)}$ are called Patankar-weights and the denominators $\sigma_i$ and $\pi_i^{(k)}$ are called Patankar-weight denominators (PWD).
\end{defn}
The following remarks  comment on the free parameters in the definition of MPRK schemes.
\begin{rem}
The parameter $\delta\in\{0,1\}$ in \eqref{eq:MPRKstages} controls the conservation of the stage values.
These are conservative if $\delta=1$, otherwise they are not, since production and destruction terms are weighted differently. See Lemma~\ref{lem:MPRKcons} below.
\end{rem}
\begin{rem}
We require $\sigma_i$ to be independent of $y_i^{n+1}$ to ensure the scheme's positivity and linear implicity. If we choose $\sigma_i=y_i^{n+1}$, we end up with the original Runge-Kutta scheme, which is not unconditionally positive. If $\sigma_i$ is a non-linear function of $y_i^{n+1}$ we would have to solve a non-linear system instead of a linear one to compute $y_i^{n+1}$. For the same reason we require $\pi_i^{(k)}$ to be independent of $y_i^{(k)}$.
\end{rem}
\begin{rem}
One might get the impression that $\sigma_i$ and $\pi_i^{(k)}$ remain constant during the time integration. As we will see, they
 are chosen as functions of stage values in all the following schemes. Thus, they will change from time step to time step. But for the sake of simplicity this will not be reflected in the notation.
\end{rem}
\begin{rem}\label{rem:negweights}
 Definition~\ref{def:MPRKdefn} is formulated for non-negative Runge-Kutta parameters.
 But MPRK schemes with negative Runge-Kutta parameters can be devised as well.
 In this case, the weighting of the production and destruction terms which get multiplied by the negative weight must be interchanged.
 This procedure will ensure the unconditional positivity of the scheme, but may have an impact on the necessary requirements to obtain a certain order of accuracy. To avoid multiple case distinctions we demand for positive Runge-Kutta parameters.
\end{rem}

Due to the introduction of the Patankar-weights, $s$ linear systems of size $N\times N$ need to be solved to obtain the stage values and the approximation at the next time level. 
In consideration of $p_{ii}=d_{ii}=0$ for $i=1,\dots,N$, the scheme \eqref{eq:MPRK} can be written in matrix-vector notation as
\begin{subequations}\label{eq:MPRKMVs}
\begin{align}\label{eq:MPRKMVstage}
 \mbfM^{(k)}\mbfy^{(k)}&=\mbfy^n+(1-\delta)\Delta t\mbfP(\mbfy^n),\quad k=1,\dots,s,\\\label{eq:MPRKMV}
 \mbfM\mbfy^{n+1}&=\mbfy^n,
\end{align}
\end{subequations}
with $\mbfP(\mbfy^n)=(P_1(\mbfy^n),\dots,P_N(\mbfy^n))^T$ and 
 \begin{equation}\label{eq:MPRKmatvecstage}
 \begin{split}
	m_{ii}^{(k)} &= 1 + \Delta t\sum_{\nu=1}^{k-1} a_{k\nu}
 \sum_{j=1}^N d_{ij}(\mbfy^{(\nu)})/\pi_i^{(k)} >0,\quad i=1,\dots, N,\\
  m_{ij}^{(k)} &=-\Delta t\delta\sum_{\nu=1}^{k-1} a_{k\nu}  p_{ij}(\mbfy^{(\nu)})/\pi_j^{(k)} \leq 0,\quad i,j=1,\dots,N,\,i\ne j,
  \end{split}
 \end{equation}
 for $k=1,\dots,s$ and
  \begin{equation}\label{MPRKmatvec}
   \begin{split}
  m_{ii} &= 1 + \Delta t\sum_{k=1}^s b_k 
 \sum_{j=1}^N d_{ij}(\mbfy^{(k)})/\sigma_i >0,\quad i=1,\dots,N,\\
  m_{ij} &=-\Delta t\sum_{k=1}^s b_k  p_{ij}(\mbfy^{(k)})/\sigma_j \leq 0,\quad i,j=1,\dots,N,\,i\ne j.
  \end{split}
 \end{equation}
If $\delta = 0$, the matrices $\mbfM^{(k)}$ become diagonal and the production terms appear on the right hand side of \eqref{eq:MPRKMVstage}.
 
The following two lemmas show that MPRK schemes as defined in Definition~\ref{def:MPRKdefn} are indeed unconditionally positive and conservative. 
Both lemmas are slight generalizations of lemmas from \cite{BDM2003}.
\begin{lem}\label{lem:MPRKcons}
 A MPRK scheme \eqref{eq:MPRK} applied to a conservative PDS is unconditionally conservative. 
 If $\delta=1$, the same holds for all stage values, this is
 $\sum_{i=1}^N (y_i^{(k)}-y_i^n)=0$
 for $k=1,\dots,s$.
\end{lem}
\begin{proof}
Since we consider a conservative PDS, we have $p_{ij}(\mbfy)=d_{ji}(\mbfy)$. Thus, we see
\begin{align*}
\sum_{i=1}^N (y_i^{n+1}-y_i^n) &= \Delta t\sum_{k=1}^s b_k\sum_{i,j=1}^N\biggl( p_{ij}(\mbfy^{(k)})\frac{y_j^{n+1}}{\sigma_j}-d_{ij}(\mbfy^{(k)})\frac{y_i^{n+1}}{\sigma_i}\biggr)\\
&= \Delta t\sum_{k=1}^s b_k\sum_{i,j=1}^N\biggl( d_{ji}(\mbfy^{(k)})\frac{y_j^{n+1}}{\sigma_j}-d_{ij}(\mbfy^{(k)})\frac{y_i^{n+1}}{\sigma_i}\biggr) = 0.
\end{align*}
The same argument can be used to show the conservation of the stages if $\delta=1$.
\end{proof}

\begin{lem}\label{lem:MPRKpos}
A MPRK scheme \eqref{eq:MPRK} is unconditionally positive. 
 The same holds for all the stages of the scheme, this is for all $\Delta t>0$ and $\mbfy^n>0$ we have $\mbfy^{(k)}>0$ for $k=1,\dots,s$.
\end{lem}
\begin{proof}
From \eqref{MPRKmatvec} we see that $m_{ii}>0$ and $m_{ij}\leq 0$ for $i,j=1,\dots,N$ with $i\ne j$. Furthermore,
\begin{align*}
\abs{m_{ii}}&=1 + \Delta t\sum_{k=1}^s b_k \sum_{j=1}^N d_{ij}(\mbfy^{(k)})/\sigma_i\\
&>\Delta t\sum_{k=1}^s b_k \sum_{j=1}^N p_{ji}(\mbfy^{(k)})/\sigma_i=\sum_{\substack{j=1\\j\ne i}}^N(-m_{ji})=\sum_{\substack{j=1\\j\ne i}}^N\abs{m_{ji}},
\end{align*}
for $i=1,\dots,N$, which shows that $\mbfM^T$ is strictly diagonally dominant. 
Altogether, $\mbfM^T$ is a M-matrix (\cite[Lemma~6.2]{Axelsson1994}) and we have $\mbfM^{-T}\geq 0$.
Thus, even $\mbfM^{-1}\geq 0$ and hence $\mbfy^{n+1}=\mbfM^{-1}\mbfy^n >0$, since $\mbfy^n>0$ and $\mbfM$ is nonsingular.

 The same argument can be applied to prove the positivity of the stage values if $\delta=1$. 
 If $\delta=0$, the system matrices in \eqref{eq:MPRKMVstage} become diagonal and positive. 
 The right hand sides $\mbfy^n+\Delta t\mbfP(\mbfy^n)$ are positive as well, since $\mbfy^n>0$ and $\mbfP(\mbfy^n)\geq 0$. 
\end{proof}

\begin{rem}
 It is worth to point out that the MPRK schemes \eqref{eq:MPRK} will generate positive solutions, even when applied to a non-positive PDS.
 It is therefore the user's responsibility to take care of this issue.
\end{rem}

Lemmas~\ref{lem:MPRKcons} and \ref{lem:MPRKpos} show that the MPRK schemes as defined in Definition~\ref{def:MPRKdefn} possess the desired properties of unconditional positivity and conservation. The only quantities left to choose are the PWDs $\sigma_i$ and $\pi_i^{(k)}$ for $i=1,\dots,N$ and $k=1,\dots,s$. In the next section we derive necessary and sufficient conditions for MPRK schemes to become first or second order accurate.
\section{Order conditions for MPRK schemes}\label{sec:MPRKorder}
In this section we assume that all occurring PDS are positive. 
To prove convergence of the MPRK schemes we investigate the local truncation errors.
In doing so, we make frequent use of the Landau symbol $\O$ and omit to specify the limit process $\Delta t\to 0$ each time.
As customary, we identify $y_i^n$ and $y_i(t^n)$ for $i=1,\dots,N$ when studying the truncation errors. 
Furthermore, since we are dealing with positive PDS we assume $y_i^n>0$ for $i=1,\dots,N$.
%
%

The following lemmas are helpful to analyze the order of MPRK schemes. 
\begin{lem}\label{lem:mbound}
 Let $\mbfM$, $\mbfM^{(k)}$ be given by \eqref{MPRKmatvec},
 \eqref{eq:MPRKmatvecstage} with $\delta = 1$ and $\mbfM^{-1}=(\widetilde m_{ij})$,  \mbox{$(\mbfM^{(k)})^{-1}=\widetilde m_{ij}^{(k)}$}. Then, we have
 \[0\leq \widetilde m_{ij},\widetilde m_{ij}^{(k)}\leq 1,\quad i,j=1,\dots,N,\]
 for $k=1,\dots,s$.
 \end{lem}
\begin{proof}
 We show the argument for the matrix $\mbfM$, the proof for $\mbfM^{(k)}$, $k=1,\dots,s$ follows the same lines.
 
 Summation of the $j$th column of $\mbfM$ and taking advantage of the property $p_{ij}=d_{ji}$ yields
 \begin{align*}
 \sum_{i=1}^N m_{ij} &= 1 + \Delta t\sum_{k=1}^s b_k 
 \sum_{\ell=1}^N d_{j\ell}(\mbfy^{(k)})/\sigma_j -  \sum_{i=1}^N\Delta t\sum_{k=1}^s b_k  p_{ij}(\mbfy^{(k)})/\sigma_j\\
 &= 1 + \Delta t\sum_{k=1}^s b_k\biggl(
 \sum_{\ell=1}^N d_{j\ell}(\mbfy^{(k)})/\sigma_j - \sum_{i=1}^N  d_{ji}(\mbfy^{(k)})/\sigma_j\biggr)\\
 &=1.
 \end{align*}
 This can also be stated as $\mbfe^T\mbfM=\mbfe^T$, with $\mbfe=(1,\dots,1)^T\in\mathbb{R}^N$ and consequently we get $\mbfe^T=\mbfe^T\mbfM^{-1}$.
 Since we know that $\mbfM^{-1}\geq 0$ from Lemma~\ref{lem:MPRKpos}, we can conclude
 \[0\leq \widetilde m_{ij}\leq 1,\quad i,j=1,\dots,N.\]
\end{proof}

\begin{lem}\label{lem:lemO}
The statement
\begin{equation}\label{eq:lemOcond1}
  \xi - \mu \eta=\O(\Delta t^s)\text{ for all $\mu>0$},
 \end{equation}
is equivalent to
\begin{equation}\label{eq:lemOcond2}
 \xi = \O(\Delta t^s) \text{ and } \eta = \O(\Delta t^s).
\end{equation}
\end{lem}
\begin{proof}
Let $\mu_1,\mu_2>0$ with $\mu_1\ne \mu_2$.
Due to \eqref{eq:lemOcond1} we have
\[\xi - \mu_1\eta=\O(\Delta t^s),\quad\xi - \mu_2\eta=\O(\Delta t^s).\]
Subtracting the first from the second equation shows
\[
(\mu_1-\mu_2)\eta=\O(\Delta t^s),
\]
which implies 
\[
\eta=\O(\Delta t^s),
\]
since $\mu_1-\mu_2\ne 0$ is constant.
Inserting this into \eqref{eq:lemOcond1}, we can conclude
\[
\xi=\O(\Delta t^s)
\]
as well. 
On the other hand, if \eqref{eq:lemOcond2} holds true, so does
$\xi-\mu\eta=\O(\Delta t^s)$ for all $\mu>0$.
\end{proof}

To derive necessary conditions that allow for a certain order of a MPRK scheme, it suffices to consider specific PDS.
In this regard, the following family of PDS will be very helpful. 
%
Given parameters $I,J\in\{1,\dots,N\}$, $I\ne J$ and $\mu>0$, we consider
\begin{subequations}\label{eq:PDSorder}
\begin{equation}
 \frac{dy_i}{dt}(t) = \PPDS_i(\mbfy(t)) - \DPDS_i(\mbfy(t)),\quad i=1,\dots,N,
\end{equation}
with
\begin{equation}
\PPDS_i(\mbfy) = \begin{cases}\mu y_I, & i=J,\\ \phantom{y}0, &\text{otherwise},\end{cases}
\quad\text{and}\quad
\DPDS_i(\mbfy) = \begin{cases}\mu y_I, & i=I,\\ \phantom{y}0, &\text{otherwise},\end{cases}
\end{equation}
\end{subequations}
 and initial values $y_i(0)=1$ for $i=1,\dots,N$.
This PDS can be written in the form
\begin{align*}
\frac{dy_I}{dt} =-\mu y_I,\qquad
\frac{dy_J}{dt} =\mu y_I,\qquad
\frac{dy_i}{dt} = 0, \quad i\in\{1,\dots,N\}\setminus\{I,J\}.
\end{align*}
and the exact solution is given by
\begin{align*}
y_I(t) = e^{-\mu t},\quad
y_J(t) = 2-e^{-\mu t},\quad
y_i(t) =1, \quad i\in\{1,\dots,N\}\setminus\{I,J\}.
\end{align*}
This shows that the PDS is positive and it is also fully conservative,
since we can write
\[
\PPDS_i(\mbfy) = \sum_{j=1}^N \pPDS_{ij}(\mbfy),\quad \DPDS_i(\mbfy) = \sum_{j=1}^N \dPDS_{ij}(\mbfy),
\]
with
\[
\pPDS_{ij}(\mbfy)=\begin{cases}\mu y_I,& i=J\text{ and }j=I,\\\phantom{y}0,&\text{otherwise},\end{cases}\quad
\dPDS_{ij}(\mbfy)=\begin{cases}\mu y_I,& i=I\text{ and }j=J,\\\phantom{y}0,&\text{otherwise}.\end{cases}
\]

\subsection{First order MPRK schemes}\label{sec:MPRKorder1}
The only first order explicit one-stage Runge-Kutta scheme is the forward Euler method, as given by the Butcher tableau
\[
\begin{array}{c|c}
 0        &   0\\\hline
          & 1
\end{array}.
\]
The corresponding MPRK scheme reads
\begin{equation}\label{eq:MPRK11}
y_i^{n+1} = y_i^n+\Delta t\sum_{j=1}^N \biggl(p_{ij}(\mbfy^n)\frac{y_j^{n+1}}{\sigma_j}-d_{ij}(\mbfy^n)\frac{y_i^{n+1}}{\sigma_i}\biggr),\quad i=1,\dots,N.
\end{equation}
In \cite{BDM2003} the choices $\sigma_i=y_i^n$ for $i=1,\dots,N$ were made to obtain a first order unconditionally positive and conservative scheme. 
The next theorem shows that this is not the only possible choice of Patankar-weights to obtain a first order scheme.
\begin{thm}\label{lem:MPRK11}
The one-stage MPRK scheme \eqref{eq:MPRK11} is first order accurate, if and only if
the conditions
\begin{equation}\label{eq:MPRK11order1}
{\sigma_i}={y_i^{n}}+\O(\Delta t),\quad i=1,\dots,N,
\end{equation}
are satisfied.
\end{thm}
\begin{proof}
For the sake of simplicity, we use the notation $\phi^*$ to represent $\phi(\mbfy^*)$ for a given function $\phi$.
The exact solution of \eqref{eq:pds} at time level $t^{n+1}$ can be expressed as
\begin{equation}\label{eq:MPRK11exact}
 y_i(t^{n+1}) = y_i^n + \Delta t(P_i^n-D_i^n) + \O(\Delta t^2)
\end{equation}
for $i=1,\dots,N$.

First, we want to derive necessary conditions, which allow for first order accuracy of the MPRK scheme \eqref{eq:MPRK11}.
For this purpose, we assume that \eqref{eq:MPRK11} is a first order scheme, this is 
\begin{equation}\label{eq:MPRK11aux2}
 y_i^{n+1} = y_i^n + \Delta t(P_i^n-D_i^n) + \O(\Delta t^2)
\end{equation}
for $i=1,\dots,N$.
From \eqref{eq:MPRK11exact} and \eqref{eq:MPRK11aux2} we find $y_i^{n+1}-y(t_i^{n+1})=\O(\Delta t^2)$ for $i=1,\dots,N$.
Utilizing \eqref{eq:MPRK11}, this can be written in the form
\begin{equation*} 
\Delta t\sum_{j=1}^N \biggl(p_{ij}^n\frac{y_j^{n+1}}{\sigma_j}-d_{ij}^n\frac{y_i^{n+1}}{\sigma_i}\biggr)-\Delta t(P_i^n-D_i^n) = \O(\Delta t^2),
\end{equation*} 
 and
further simplifications yield
\begin{equation}\label{eq:MPRK11aux1} 
\sum_{j=1}^N \biggl(p_{ij}^n\biggl(\frac{y_j^{n+1}}{\sigma_j}-1\biggr)-d_{ij}^n\biggl(\frac{y_i^{n+1}}{\sigma_i}-1\biggr)\biggr) = \O(\Delta t)
\end{equation}
for $i=1,\dots,N$.
Now we assume that the scheme is used to solve the PDS \eqref{eq:PDSorder} with parameters $I,J\in\{1,\dots,N\}$, $I\ne J$ and $\mu=1$.
In this case, equation $\eqref{eq:MPRK11aux1}$ with $i=I$ becomes
\begin{equation*}
-y_I^n \biggl(\frac{y_I^{n+1}}{\sigma_I}-1\biggr)=\O(\Delta t).
\end{equation*}
and can be rewritten as
\begin{equation}\label{eq:MPRK11aux3}
\frac{y_I^{n+1}}{\sigma_I}=1+\O(\Delta t).
\end{equation}
Since $y_I^{n+1}\to y_I^n$ due to \eqref{eq:MPRK11aux2}, it follows that $\sigma_I\to y_I^n$ as well.
Thus, substitution of $\eqref{eq:MPRK11aux2}$ into \eqref{eq:MPRK11aux3}  yields
\begin{equation*}
 \sigma_I = y_I^n+\O(\Delta t), 
\end{equation*}
since $y_I^n>0$.
As $I\in\{1,\dots,N\}$ was chosen arbitrary, we see that \eqref{eq:MPRK11order1} is necessary for first order accuracy.
 
Now we show that \eqref{eq:MPRK11order1} is also sufficient to obtain a first order scheme.
Expressing one step of the scheme \eqref{eq:MPRK11} using \eqref{eq:MPRKMV}, and utilizing \eqref{eq:MPRK11order1} and Lemma~\ref{lem:mbound}, we see
\[
\frac{y_i^{n+1}}{\sigma_i}=\sum_{j=1}^N \underbrace{\widetilde m_{ij}}_{\in[0,1]} \frac{y_j^n}{\sigma_i}=\O(1),
\]
for $i=1,\dots,N$ with $(\widetilde m_{ij})=\mbfM^{-1}\in\R^{N\times N}$.
Consequently, due to \eqref{eq:MPRK11} we have
\begin{equation*}
 y_i^{n+1} = y_i^n+\Delta t\sum_{j=1}^N \biggl(p_{ij}^n\frac{y_j^{n+1}}{\sigma_j}-d_{ij}^n\frac{y_i^{n+1}}{\sigma_i}\biggr)=y_i^n+\O(\Delta t)
\end{equation*}
for $i=1,\dots,N$.
Substituting this into \eqref{eq:MPRK11} and using \eqref{eq:MPRK11order1} yields
\begin{equation*}
 y_i^{n+1} = y_i^n+\Delta t\sum_{j=1}^N \biggl(p_{ij}^n\frac{y_j^n}{\sigma_j}-d_{ij}^n\frac{y_i^n}{\sigma_i}\biggr)+\O(\Delta t^2)
\end{equation*}
for $i=1,\dots,N$.
Finally, according to \eqref{eq:MPRK11order1} and \eqref{eq:MPRK11exact} we get
\begin{align*}
 y_i^{n+1} &\eq^{\eqref{eq:MPRK11order1}} y_i^n+\Delta t\sum_{j=1}^N \left(p_{ij}^n(1+\O(\Delta t))-d_{ij}^n(1+\O(\Delta t))\right)+\O(\Delta t^2)\\
 &=y_i^n+\Delta t\sum_{j=1}^N \left(p_{ij}^n-d_{ij}^n\right)+\O(\Delta t^2)\\
 &=y_i^n+\Delta t(P_i^n-D_i^n)+\O(\Delta t^2)\\
 &\eq^{\eqref{eq:MPRK11exact}}y(t_i^{n+1})+\O(\Delta t^2)
\end{align*}
for $i=1,\dots,N$.
Hence, condition \eqref{eq:MPRK11order1} suffices to obtain a first order scheme. 
\end{proof}
The theorem shows, that the choice $\sigma_i=y_i^n$ for $i=1,\dots,N$ as made in \cite{BDM2003} seems likely, but is not necessary.
The corresponding MPRK scheme reads
\begin{equation}\label{eq:MPE}
 y_i^{n+1} = y_i^n+\Delta t\sum_{j=1}^N \biggl(p_{ij}(\mbfy^n)\frac{y_j^{n+1}}{y_j^n}-d_{ij}^n(\mbfy^n)\frac{y_i^{n+1}}{y_i^n}\biggr)
\end{equation}
for $i=1,\dots,N$ and was named modified Patankar-Euler (MPE) scheme.

We can use the additional degree of freedom to design methods which minimize the truncation error or are even of second order for specific differential equations.
For instance, the choice $\sigma_i=y_i^n(1-3\Delta t)$ for $i=1,\dots,N$  results in a second order scheme for the linear test problem \eqref{eq:lintest} of Section~\ref{sec:testcases}.
Unfortunately, the resulting scheme is not a MPRK scheme, since $\sigma_i$ becomes non-positive for $\Delta t\geq 1/3$. To overcome this issue, we can define
\begin{equation}\label{eq:sigMPElin}
\sigma_i=\begin{cases}y_i^n(1-3\Delta t), & \Delta t< 1/3,\\ y_i^n, & \Delta t \geq 1/3,\end{cases}\quad i=1,\dots,N.
\end{equation}
We will refer to the scheme \eqref{eq:MPRK11} with Patankar-weights \eqref{eq:sigMPElin} as MPElin.
Numerical results demonstrating the scheme's improved accuracy can be found in Section~\ref{sec:numres}.

In complex applications MPRK schemes are usually used as time integrators of biogeochemical submodels.
The above example shows that it may be fruitful to search for optimal PWDs for a specific submodel, as slight changes of an existing code might really improve accuracy.

The same ideas could even be used to minimize truncation errors or possibly improve the order of higher order MPRK schemes. 
However, in order to focus on a general investigation of MPRK schemes, we don't pursue this idea any further in this paper.
\subsection{Second order MPRK schemes}\label{sec:MPRKorder2}
The second order MPRK scheme introduced in \cite{BDM2003} is a modification of Heun's method. 
In this section we will show how MPRK schemes based on general explicit second order two-stage Runkge-Kutta schemes can be developed.

A MPRK scheme \eqref{eq:MPRK} with two stages reads
\begin{subequations}\label{eq:MPRK22}
\begin{align}
  &\begin{aligned}
    \mathllap{y_i^{(1)}} &= y_i^n,
  \end{aligned}\\ \label{eq:MPRK2c2}
  &\begin{aligned}
    \mathllap{y_i^{(2)}} &= y_i^n + a_{21}\Delta t\sum_{j=1}^N\biggl(p_{ij}(\mbfy^{(1)})(1-\delta)+ p_{ij}(\mbfy^{(1)})\frac{y_j^{(2)}}{\pi_j}\delta-d_{ij}(\mbfy^{(1)})\frac{y_i^{(2)}}{\pi_i}\biggr),
  \end{aligned}\\ \label{eq:MPRK2cnp1}
  &\begin{multlined}[b][.7\columnwidth]
    \mathllap{y_i^{n+1}} = y_i^n + \Delta t\sum_{j=1}^N\biggl( \bigl(b_1 p_{ij}(\mbfy^{(1)})+b_2 p_{ij}(\mbfy^{(2)})\bigr)\frac{y_j^{n+1}}{\sigma_j}\\
         - \bigl(b_1 d_{ij}(\mbfy^{(1)})+ b_2 d_{ij}(\mbfy^{(2)})\bigr)\frac{y_i^{n+1}}{\sigma_i}\biggr),
  \end{multlined}  
\end{align}
\end{subequations}
for $i=1,\dots,N$ with $\delta\in\{0,1\}$.
In this setting, the original MPRK scheme introduced in \cite{BDM2003} is obtained by setting $\delta = 1$, $\pi_i=y_i^n$ and $\sigma_i=y_i^{(2)}$ for $i=1,\dots,N$.

The next theorem presents necessary and sufficient conditions for second order accuracy of two-stage MPRK schemes. 
\begin{thm}\label{thm:MPRK22}
 Given non-negative parameters of an explicit second order Runge-Kutta scheme, this is \begin{equation*}
b_1+b_2=1,\quad a_{21}b_2=\frac12,
 \end{equation*}
 the MPRK scheme \eqref{eq:MPRK22} is of second order, if and only if the conditions
\begin{subequations}\label{eq:MPRK22order_neu}
\begin{equation}\label{eq:MPRK22order1_neu}
{\pi_i} = {y_i^n}+\O(\Delta t), \quad i=1,\dots,N,
\end{equation}
and
\begin{equation}\label{eq:MPRK22order2_neu}
{\sigma_i}={y_i^n+\Delta t (P_i^n-D_i^n)}+\O(\Delta t^2),\quad i=1,\dots,N
\end{equation}
\end{subequations}
are satisfied.
\end{thm}
\begin{proof}
We use the notation $\phi^*$ to represent $\phi(\mbfy^*)$ for a given function $\phi$.
Since $y_i'(t^n)=P_i^n-D_i^n$ implies 
\[y_i''(t^n)=\frac{\partial(P_i^n-D_i^n)}{\partial\mbfy}\mbfy'(t^n)=\frac{\partial(P_i^n-D_i^n)}{\partial\mbfy}(\mbfP^n-\mbfD^n),\]
with  $\mbfP=(P_1,\dots,P_N)^T$, $\mbfD=(D_1,\dots,D_N)^T$, it follows, that the exact solution at time $t^{n+1}$ can be written as
\begin{equation}\label{eq:MPRK22exact}
y_i(t^{n+1})=y_i^n+\Delta t(P_i^n-D_i^n)+\frac{\Delta t^2}{2} \frac{\partial(P_i^n-D_i^n)}{\partial\mbfy}(\mbfP^n-\mbfD^n)+\O(\Delta t^3)
\end{equation}
for $i=1,\dots,N$.

 First, we derive necessary conditions, which allow \eqref{eq:MPRK22} to become a second order scheme. To do so, we assume that \eqref{eq:MPRK22} is of second order, and consequently
 \begin{equation}\label{eq:MPRK22numexp}
 y_i^{n+1}=y_i^n+\Delta t(P_i^n-D_i^n)+\frac{\Delta t^2}{2} \frac{\partial(P_i^n-D_i^n)}{\partial\mbfy}(\mbfP^n-\mbfD^n)+\O(\Delta t^3)
 \end{equation}
 for $i=1,\dots,N$.
 Due to \eqref{eq:MPRK22exact} and \eqref{eq:MPRK22numexp} we see
 $y_i^{n+1}-y_i(t^{n+1})=\O(\Delta t^3)$, which, according to \eqref{eq:MPRK2cnp1} and \eqref{eq:MPRK22exact}, can be written in the form
 \begin{multline}\label{eq:MPRK22aux1}
  \Delta t\sum_{j=1}^N\biggl( (b_1 p_{ij}^n+b_2 p_{ij}^{(2)})\frac{y_j^{n+1}}{\sigma_j}
         - (b_1 d_{ij}^n+ b_2 d_{ij}^{(2)})\frac{y_i^{n+1}}{\sigma_i}\biggr)\\
-\Delta t(P_i^n-D_i^n)-\frac{\Delta t^2}{2} \frac{\partial(P_i^n-D_i^n)}{\partial\mbfy}(\mbfP^n-\mbfD^n)=\O(\Delta t^3)       
 \end{multline}
for $i=1,\dots,N$.
From now on, we focus on the solution of the PDS \eqref{eq:PDSorder}
with parameters $I,J\in\{1,\dots,N\}$, $I\ne J$ and $\mu>0$.
For this PDS, the equation \eqref{eq:MPRK22aux1} with $i=I$ becomes
 \begin{equation}\label{eq:MPRK22aux1b}
         - (b_1 \DPDS_I^n+ b_2 \DPDS_I^{(2)})\frac{y_I^{n+1}}{\sigma_I}
+\DPDS_I^n-\frac{\Delta t}{2} \frac{\partial \DPDS_I^n}{\partial y_I}\DPDS_I^n =\O(\Delta t^2).       
 \end{equation}
Since the PDS \eqref{eq:PDSorder} is linear, we have
 \begin{equation*}
\DPDS_{I}^{(2)}=\DPDS_{I}^n + \frac{\partial \DPDS_{I}^n}{\partial y_I}(y_I^{(2)}-y_I^n),
\end{equation*}
as derivatives of order two and higher vanish.
Substituting this into \eqref{eq:MPRK22aux1b} 
yields
%
  \begin{equation*}
   -\DPDS_{I}^n\biggl((\overbrace{b_1+b_2}^{=1})\frac{y_I^{n+1}}{\sigma_I}-1\biggr)
  -\frac{\partial\DPDS_{I}^n}{\partial y_I}\biggl(b_2(y_I^{(2)}-y_I^n)\frac{y_I^{n+1}}{\sigma_I}+\frac{\Delta t}2 \DPDS_{I}^n\biggr)
 = \O(\Delta t^2).     
 \end{equation*}
Insertion of \eqref{eq:MPRK2c2} and taking account of $a_{21}b_2=1/2$ results in
  \begin{equation*}
   -\DPDS_{I}^n\biggl(\frac{y_I^{n+1}}{\sigma_I}-1\biggr)
  -\underbrace{\frac{\partial\DPDS_{I}^n}{\partial y_I}}_{=\mu}\biggl(- \frac{\Delta t}2 \DPDS_{I}^n\frac{y_I^{(2)}}{\pi_I}\frac{y_I^{n+1}}{\sigma_I}+\frac{\Delta t}2 \DPDS_{I}^n\biggr)
 = \O(\Delta t^2),     
 \end{equation*}
irrespective of the value of $\delta$.
Owing to $\DPDS_I^n=\mu y_I^n>0$, this can be further simplified to
  \begin{equation*}
  \frac{y_I^{n+1}}{\sigma_I}-1 
  -\mu\frac{\Delta t}2\biggl(\frac{y_I^{(2)}}{\pi_I}\frac{y_I^{n+1}}{\sigma_I}-1\biggr)
 = \O(\Delta t^2).     
 \end{equation*}
Since $\mu>0$ was chosen arbitrary, we can conclude that the above equation holds for all $\mu>0$. 
Hence, by Lemma~\ref{lem:lemO} we can conclude
\begin{equation}\label{eq:MPRK22aux3}
 \frac{y_I^{n+1}}{\sigma_I} =1+ \O(\Delta t^2)
\end{equation}
and
\begin{equation}\label{eq:MPRK22aux4}
 \frac{y_I^{(2)}}{\pi_I}\frac{y_I^{n+1}}{\sigma_I}=1+\O(\Delta t).
\end{equation}
Since $y_I^{n+1}=y_I^n+\Delta t( \PPDS_I^n-\DPDS_I^n)+\O(\Delta t^2)$ owing to \eqref{eq:MPRK22numexp}, we find $\sigma_I\to y_I^n$ and
\begin{equation*}
 {\sigma_I}={y_I^n+\Delta t(\PPDS_I^n-\DPDS_I^n)}+\O(\Delta t^2)
\end{equation*}
from \eqref{eq:MPRK22aux3}. 
Inserting $y_I^{n+1}/\sigma_I\to 1$ in \eqref{eq:MPRK22aux4} shows 
\begin{equation}\label{eq:MPRK22aux13}
y_I^{(2)}/\pi_I\to 1
\end{equation}
as well.
Utilizing this in \eqref{eq:MPRK2c2} we find
\begin{equation}\label{eq:MPRK22aux6}
 y_I^{(2)}=y_I^n-a_{21}\Delta t \DPDS_{I}^n \underbrace{\frac{y_I^{(2)}}{\pi_I}}_{=\O(1)}=y_I^n+\O(\Delta t).
\end{equation}
and substituting this into \eqref{eq:MPRK22aux13} implies $\pi_I\to y_I^n$.
Altogether, we see
\begin{equation*}
 1+\O(\Delta t)\eq_{\eqref{eq:MPRK22aux3}}^{\eqref{eq:MPRK22aux4}}=\frac{y_I^{(2)}}{\pi_I}(1+\O(\Delta t^2))=\frac{y_I^{(2)}}{\pi_I}+\O(\Delta t^2)\eq^{\eqref{eq:MPRK22aux6}} \frac{y_I^n}{\pi_I}+\O(\Delta t).
\end{equation*}
and hence
\begin{equation*}
 {\pi_I}={y_I^n}+\O(\Delta t).
\end{equation*}
As $I\in\{1,\dots,N\}$ was chosen arbitrary, we can let it run from $1,\dots,N$ and find that \eqref{eq:MPRK22order1_neu} and \eqref{eq:MPRK22order2_neu} are indeed necessary conditions.

Next, we show that the conditions \eqref{eq:MPRK22order_neu} are already sufficient to obtain a second order scheme.
For the sake of clarity, we start considering MPRK schemes with $\delta=1$. 
The MPRK scheme \eqref{eq:MPRK22} can be written as two linear systems
\[
\mbfM^{(2)}\mbfy^{(2)}=\mbfy^n,\quad \mbfM\mbfy^{n+1}=\mbfy^{n},
\]
and, since we assume $\delta=1$, we know from Lemma~\ref{lem:mbound} that $\mbfM^{-1}=\O(1)$ and $(\mbfM^{(2)})^{-1}=\O(1)$.
Thus, we have $\mbfy^{(2)}=\O(1)$ and $\mbfy^{n+1}=\O(1)$ and conditions \eqref{eq:MPRK22order1_neu} and \eqref{eq:MPRK22order2_neu}
lead to
\begin{equation}\label{eq:MPRK22aux7}
 \frac{y_i^{(2)}}{\pi_i}=\O(1)
\end{equation}
and 
\begin{equation}\label{eq:MPRK22aux8}
 \frac{y_i^{n+1}}{\sigma_i}=\O(1)
\end{equation}
for $i=1,\dots,N$, since $y_i^n>0$.
The boundedness of the Patankar-weights \eqref{eq:MPRK22aux7} shows that \eqref{eq:MPRK2c2} yields
\begin{equation}\label{eq:MPRK22aux9}
y_i^{(2)}=y_i^n + a_{21}\Delta t\sum_{j=1}^N\biggl( p_{ij}^n\frac{y_j^{(2)}}{\pi_j}-d_{ij}^n\frac{y_i^{(2)}}{\pi_i}\biggr)=y_i^n+\O(\Delta t)
\end{equation}
for $i=1,\dots,N$.
Inserting this and \eqref{eq:MPRK22order1_neu} into \eqref{eq:MPRK2c2} shows
\begin{equation*}
 y_i^{(2)}=y_i^n+a_{21}\Delta t\sum_{j=1}^N\left(p_{ij}^n(1+\O(\Delta t))-d_{ij}^n(1+\O(\Delta t))\right)
\end{equation*}
and further
\begin{equation}\label{eq:MPRK22aux12}
 y_i^{(2)}=y_i^n+a_{21}\Delta t(P_i^n-D_i^n)+\O(\Delta t^2)
\end{equation}
for $i=1,\dots,N$.
Now we compute an expansion of $y_i^{n+1}$ using \eqref{eq:MPRK2cnp1}. Since $\mbfy^{(2)}-\mbfy^n=\O(\Delta t)$ according to \eqref{eq:MPRK22aux9} we get
 \begin{multline}\label{eq:MPRK22aux11}
  y_i^{n+1}=y_i^n+\Delta t\sum_{j=1}^N \biggl(\biggl(b_1 p_{ij}^n +b_2\biggl( p_{ij}^{n}+\frac{\partial p_{ij}^n}{\partial \mbfy}(\underbrace{\mbfy^{(2)}-\mbfy^n}_{=\O(\Delta t)})+\O(\Delta t^2)\biggr)\biggr)\frac{y_j^{n+1}}{\sigma_j}\biggr.\\
  -\biggl(b_1 d_{ij}^n +b_2\biggl( d_{ij}^{n}+\frac{\partial d_{ij}^n}{\partial \mbfy}(\underbrace{\mbfy^{(2)}-\mbfy^n}_{=\O(\Delta t)})+\O(\Delta t^2)\biggr)\biggr)\frac{y_i^{n+1}}{\sigma_i}\biggl.\biggr)
 \end{multline}
 and with \eqref{eq:MPRK22aux8} we can conclude
 \begin{equation}\label{eq:MPRK22aux10}
  y_i^{n+1}=y_i^n+\O(\Delta t)
 \end{equation}
 for $i=1,\dots,N$.
 From \eqref{eq:MPRK22order2_neu} and \eqref{eq:MPRK22aux10} it follows that
 \begin{equation}\label{eq:MPRK22aux14}
  \frac{y_i^{n+1}}{\sigma_i}=1+\O(\Delta t),
 \end{equation}
which can be utilized together with $b_1+b_2=1$ in \eqref{eq:MPRK22aux11} to obtain
\begin{equation*}
 y_i^{n+1}=y_i^n+\Delta t(P_i^n-D_i^n)+\O(\Delta t^2)
\end{equation*}
for $i=1,\dots,N$.
Due to this expression, we can tighten \eqref{eq:MPRK22aux14} in the form
\begin{equation*}
 \frac{y_i^{n+1}}{\sigma_i}=1+\O(\Delta t^2)
\end{equation*}
and 
inserting this together with \eqref{eq:MPRK22aux12} and $b_2 a_{21}=1/2$ into \eqref{eq:MPRK22aux11} shows
\begin{equation*}
 y_i^{n+1}=y_i^n+\Delta t(P_i^n-D_i^n)+\frac{\Delta t^2}{2}\frac{\partial(P_i^n-D_i^n)}{\mbfy}(\mbfP^n-\mbfD^n)+\O(\Delta t^3)
\end{equation*}
for $i=1,\dots,N$.
A comparison with \eqref{eq:MPRK22exact} shows $y_i^{n+1}=y(t_i^{n+1})+\O(\Delta t^3)$ for $i=1,\dots,N$.
Thus, the MPRK scheme \eqref{eq:MPRK22} is second order accurate, if $\delta=1$.

If $\delta=0$ instead, only $y_i^{(2)}$ is affected by this change for $i=1,\dots,N$ and we need to show that \eqref{eq:MPRK22aux12} remains valid.
In this case, \eqref{eq:MPRK2c2} can be rewritten in the form
\[
y_i^{(2)}=\frac{y_i^n+a_{21}\Delta t P_i^n}{1+a_{21}\Delta t D_i^n/\pi_i}
\]
and thus,
\[
\frac{y_i^{(2)}}{\pi_i}=\frac{y_i^n+a_{21}\Delta t P_i^n}{\pi_i+a_{21}\Delta t D_i^n}\eq^{\eqref{eq:MPRK22order1_neu}}\frac{y_i^n+\O(\Delta t)}{y_i^n+\O(\Delta t)}=1+\O(\Delta t),
\]
for $i=1,\dots,N$, since $y_i^n>0$.
Inserting this into \eqref{eq:MPRK2c2} shows                                                                                                                  
\begin{equation*}                                                                                                              
y_i^{(2)}=y_i^n+a_{21}\Delta t(P_i^n-D_i^n)+\O(\Delta t^2).
\end{equation*}
Thus, \eqref{eq:MPRK22aux12} holds for $\delta=0$ as well and conditions \eqref{eq:MPRK22order_neu} suffice to make \eqref{eq:MPRK22} a second order accurate scheme irrespective of the value of $\delta$.
\end{proof}

One conclusion we can draw from Theorem~\ref{thm:MPRK22}, is that the choice of PWDs used in \cite{BDM2003} results in a second order scheme if and only if $a_{21}=1$, see Theorem~\ref{thm:MPRK22org}.
For other values of $a_{21}$ the PWDs must be chosen differently and  
Theorem~\ref{thm:MPRK22b} introduces one specific choice of PWDs that can be used with general second order explicit two-stage Runge-Kutta schemes, which have non-negative Runge-Kutta parameters.
\begin{thm}\label{thm:MPRK22org}
 Assuming an underlying second order Runge-Kutta method, the MPRK scheme \eqref{eq:MPRK22} with PWDs $\pi_i=y_i^n$ and $\sigma_i=y_i^{(2)}$ for $i=1,\dots,N$ is second order accurate, if and only if $a_{21}=1$.
\end{thm}
\begin{proof}
Obviously \eqref{eq:MPRK22order1_neu} is satisfied and due to \eqref{eq:MPRK22aux12} from Theorem~\ref{thm:MPRK22} we know
\begin{equation*}
\sigma_i={y_i^{(2)}}={y_i^n+a_{21}\Delta t(P_i^n-D_i^n)}+\O(\Delta t^2)
\end{equation*}
for $i=1,\dots,N$, irrespective of the value of $\delta$.
A comparison with condition \eqref{eq:MPRK22order2_neu} proves the statement.
\end{proof}
The following theorem shows how the ideas of \cite{BDM2003} can be generalized to obtain second order MPRK schemes for appropriate parameters $a_{21}\ne1$.
\begin{thm}\label{thm:MPRK22b}
Assuming an underlying second order Runge-Kutta scheme, the MPRK scheme \eqref{eq:MPRK22} with PWDs
\begin{equation}\label{eq:MPRK22bPW}
\pi_i=y_i^n,\quad\sigma_i=y_i^n\left(\frac{y_i^{(2)}}{y_i^n}\right)^{\!\!1/a_{21}},\quad i=1,\dots,N,
\end{equation}
is second order accurate.
\end{thm}
\begin{proof}
First of all, $\pi_i>0$ and $\sigma_i>0$ for $i=1,\dots,N$, since $y_i^n>0$ and $y_i^{(2)}>0$ for $i=1,\dots,N$ according to Lemma~\ref{lem:MPRKpos}.
Hence, \eqref{eq:MPRK22} with Patankar-weights \eqref{eq:MPRK22bPW} is indeed a MPRK scheme.

Now, we investigate the order of the scheme. 
Obviously \eqref{eq:MPRK22order1_neu} holds and
from \eqref{eq:MPRK22aux12} of Theorem~\ref{thm:MPRK22} we know
$y_i^{(2)}=y_i^n+a_{21}\Delta t(P_i^n-D_i^n)+\O(\Delta t^2)$ for $i=1,\dots,N$.
Thus, 
 \[\Bigl(y_i^{(2)}\Bigr)^{\!s}=\left(y_i^n\right)^s+s a_{21}\Delta t(P_i^n-D_i^n)\left(y_i^n\right)^{s-1}+\O(\Delta t^2),
 \]
which implies
\begin{equation*}
 \frac{\Bigl(y_i^{(2)}\Bigr)^{\!s}}{\left(y_i^n\right)^{s-1}}=y_i^n+s a_{21}\Delta t(P_i^n-D_i^n)+\O(\Delta t^2)
\end{equation*}
for $i=1,\dots,N$.
With $s=1/a_{21}$ we see $\sigma_i=y_i^n+\Delta t(P_i^n-D_i^n)+\O(\Delta t^2)$.
Hence, condition \eqref{eq:MPRK22order2_neu} is also satisfied.
 \end{proof}
Assuming $\alpha\ne 0$, a general explicit two-stage Runge-Kutta scheme of second order is given by the Butcher tableau
\[
\begin{array}{c|cc}
 0        &                    & \\
 \alpha & \alpha            & \\\hline
          & 1-\frac1{2\alpha} & \frac1{2\alpha}
\end{array},
\]
see \cite{Butcher2008}.
To make the scheme \eqref{eq:MPRK22} a MPRK scheme, we also have to ensure non-negativity of the Runge-Kutta parameters $a_{21}=\alpha$, $b_1=1-1/(2\alpha)$, and $b_2=1/(2\alpha)$.
Thus, we have to restrict $\alpha$ to $\alpha\geq 1/2$.
Prominent examples are Heun's method ($\alpha = 1$), Ralston's method ($\alpha=2/3$), and the midpoint method ($\alpha=\frac12$).

For $\alpha \geq 1/2$ Theorem~\ref{thm:MPRK22b} introduces a one parameter family of second order two-stage MPRK schemes.
\begin{subequations}\label{eq:MPRK22b}
\begin{align}
  &\begin{aligned}
    \mathllap{y_i^{(1)}} &= y_i^n,
  \end{aligned}\\ 
  &\begin{aligned}
    \mathllap{y_i^{(2)}} &= y_i^n + \alpha\Delta t\sum_{j=1}^N\biggl( p_{ij}(\mbfy^{(1)})(1-\delta)+p_{ij}(\mbfy^{(1)})\frac{y_j^{(2)}}{y_j^{(1)}}\delta-d_{ij}(\mbfy^{(1)})\frac{y_i^{(2)}}{y_i^{(1)}}\biggr),
  \end{aligned}\\ 
  &\begin{multlined}[b][.7\columnwidth]
    \mathllap{y_i^{n+1}} = y_i^n + \Delta t\sum_{j=1}^N\biggl( \left(\left(1-\frac1{2\alpha}\right) p_{ij}(\mbfy^{(1)})+\frac1{2\alpha} p_{ij}(\mbfy^{(2)})\right)\frac{y_j^{n+1}}{(y_j^{(2)})^{1/\alpha}(y_j^{(1)})^{1-1/\alpha}}\\
         - \left(\left(1-\frac1{2\alpha}\right) d_{ij}(\mbfy^{(1)})- \frac1{2\alpha} d_{ij}(\mbfy^{(2)})\right)\frac{y_i^{n+1}}{(y_i^{(2)})^{1/\alpha}(y_i^{(1)})^{1-1/\alpha}}\biggr),
  \end{multlined}  
\end{align}
\end{subequations}
for $i=1,\dots,N$.
In the following, we will refer to this family of schemes as MPRK22($\alpha$) schemes if $\delta=1$ and MPRK22ncs($\alpha$) if $\delta=0$.

Numerical experiments which confirm the theoretical convergence order and also compare the truncation errors of MPRK22($\alpha$) and MPRK22ncs($\alpha$) are presented in Section~\ref{sec:numres}.
The second order MPRK scheme introduced in \cite{BDM2003} is equivalent to MPRK22(1).

The PWDs \eqref{eq:MPRK22bPW} are not the only possible choices. Of course, many other second order MPRK schemes can be devised.
In particular, we can use convex combinations of terms like $y_i^n(y_i^{(2)}/y_i^n)^s$ to find other second order MPRK schemes.
For instance, the PWDs 
\[\pi_i=y_i^n,\quad \sigma_i=\omega y_i^n\left(\frac{y_i^{(2)}}{y_i^n}\right)^{s_1}+ (1-\omega) y_i^n\left(\frac{y_i^{(2)}}{y_i^n}\right)^{s_2},\quad i=1,\dots,N,\]
with $0\leq \omega\leq 1$, result in a two-parameter family of second order schemes, if
\[s_2 = \frac{a_{21}\omega s_1-1}{a_{21}(\omega-1)}.\]
In our numerical experiments in Section~\ref{sec:numres} we only consider the MPRK22($\alpha$) and MPRK22ncs($\alpha$) schemes, as these only contain a single  free parameter.
\section{Test problems}\label{sec:testcases}
For our numerical experiments, we consider the same three test cases as in \cite{BDM2003}.
A simple linear test problem for which the analytical solution is known, a non-stiff nonlinear test problem and the stiff Robertson problem.
Additionally, we apply the MPRK schemes to the original Brusselator problem \cite{LefeverNicolis1971}, which was used in \cite{BonaventuraDellaRocca2016} to demonstrate the workload efficiency of the MPRK22(1) scheme.

\subsection*{Linear test problem}
The simple linear test case is given by
\begin{equation}\label{eq:lintest}
 y_1'(t) = y_2(t) - a y_1(t),\quad y_2'(t) = a y_1(t) - y_2(t),
\end{equation}
with a constant parameter $a$ and initial values $y_1(0) = y_1^0$ and $y_2(0)=y_2^0$. 
We can write the right hand side in the form \eqref{eq:pijdij} with
\begin{align*}
p_{12}(\mbfy) &= y_2, & p_{21}(\mbfy) &= a y_1,\\
d_{12}(\mbfy) &= a y_1, & d_{21}(\mbfy) &= y_2,
\end{align*}
and $p_{ii}(\mbfy)=d_{ii}(\mbfy)=0$ for $i=1,2$.
The system describes exchange of mass between to constituents. The analytical solution is
\[
y_1(t) = (1+c\exp(-(a+1)t))y_1^\infty
\]
with the asymptotic solution
\[
y_1^\infty = \frac{y_1^0+y_2^0}{a+1},\quad c = \frac{y_1^0}{y_1^\infty}-1.
\]
The system is conservative and we get
\[
y_2(t) = y_1^0+y_2^0 - y_1(t).
\]
In the numerical simulations of Section~\ref{sec:numres} we use $a=5$ and initial values $y_1^0=0.9$ and $y_2^0=0.1$.
The solution is approximated on the time interval $[0,1.75]$.
\subsection*{Nonlinear test problem}
The non-stiff nonlinear test problem reads
\begin{equation}\label{eq:nonlintest}
 \begin{split}
  y_1'(t) &= -\frac{y_1(t)y_2(t)}{y_1(t)+1},\\
  y_2'(t) &= \frac{y_1(t)y_2(t)}{y_1(t)+1}-a y_2(t),\\
  y_3'(t) &= a y_2(t),
 \end{split}
\end{equation}
with initial conditions $y_i(0)=y_i^0$ for $i=1,2,3$.
To express the right hand side in the form \eqref{eq:pijdij} we can use
\begin{align*}
 p_{21}(\mbfy)=d_{21}(\mbfy)= \frac{y_1y_2}{y_1+1},\quad p_{32}(\mbfy)=d_{23}(\mbfy)=a y_2,
\end{align*}
and $p_{ij}(\mbfy)=d_{ij}(\mbfy)=0$ for all other combinations of $i$ and $j$.

The system represents a biogeochemical model for the description of an algal bloom, that transforms nutrients ($y_1$) via phytoplankton ($y_2$) into detritus ($y_3$).
In the numerical simulations of Section~\ref{sec:numres} we use the initial conditions $y_1^0=9.98$, $y_2^0=0.01$ and $y_3^0=0.01$.
The solution is approximated on the time interval $[0,30]$.
\subsection*{Original Brusselator test problem}
As another non-stiff nonlinear test case we consider the original Brusselator problem \cite{LefeverNicolis1971, HNW1993}
\begin{equation}\label{eq:brusselator}
 \begin{split}
  y_1'(t) &= -k_1 y_1(t),\\
  y_2'(t) &= -k_2 y_2(t) y_5(t),\\
  y_3'(t) &= k_2 y_2(t) y_5(t),\\
  y_4'(t) &= k_4 y_5(t),\\
  y_5'(t) &= k_1 y_1(t) - k_2 y_2(t) y_5(t) +k_3 y_5(t)^2y_6(t)-k_4 y_5(t),\\
  y_6'(t) &= k_2 y_2(t) y_5(t) - k_3 y_5(t)^2 y_6(t),
 \end{split}
\end{equation}
with constant parameters $k_i$ and initial values $y_i(0)=y_i^0$ for $i=1,\dots,6$.
The system can be written in the form \eqref{eq:pijdij}, setting
\begin{align*}
 p_{32}(\mbfy) &=d_{23}(\mbfy)= k_2 y_2 y_5, & p_{45}(\mbfy)&=d_{54}(\mbfy)=k_4 y_5, & p_{51}(\mbfy)&=d_{15}(\mbfy)=k_1 y_1,\\
 p_{56}(\mbfy)&=d_{65}(\mbfy)=k_3 y_5^2 y_6, & p_{65}(\mbfy)&=d_{56}(\mbfy)=k_2 y_2 y_5,
\end{align*}
and $p_{ij}(\mbfy)=d_{ji}(\mbfy)=0$ for all other combinations of $i$ and $j$.

In the numerical simulations of Section~\ref{sec:numres} we set $k_i=1$ for $i=1,\dots,6$ and the initial values $y_1(0)=y_2(0)=10$, $y_3(0)=y_4(0)=\mathtt
{eps}\approx2.2204\cdot10^{-16}$, and $y_5(0)=y_6(0)=0.1$.
The time interval of interest is $[0,10]$.
\subsection*{Robertson test problem}
To demonstrate the practicability of MPRK schemes in the case of stiff systems, we apply the schemes to the Robertson test case, which is given by
\begin{equation}\label{eq:robertson}
 \begin{split}
  y_1'(t) &= 10^4 y_2(t)y_3(t) - 0.04 y_1(t),\\
  y_2'(t) & = 0.04 y_1(t) - 10^4 y_2(t) y_3(t) - 3\cdot 10^7 y_2(t)^2,\\
  y_3'(t) &= 3\cdot 10^7 y_2(t)^2,
 \end{split}
\end{equation}
with initial values $y_i(0)=y_i^0$ for $i=1,2,3$.
For this problem the production and destruction rates \eqref{eq:pijdij} are given by
\[p_{12}(\mbfy)=d_{21}(\mbfy)=10^4y_2y_3,\quad p_{21}(\mbfy)=d_{12}(\mbfy)=0.04y_1,\quad p_{32}(\mbfy)=d_{23}(\mbfy)=3\cdot10^7 y_2,\]
and $p_{ij}(\mbfy)=d_{ij}(\mbfy)=0$ for all other combinations of $i$ and $j$.

We use the initial values $y_1(0)=y_1^0=1-2 \mathtt
{eps}$ and $y_2^0=y_3^0=\mathtt
{eps}\approx2.2204\cdot10^{-16}$ in the numerical simulations of Section~\ref{sec:numres}.

In this problem the reactions take place on very different time scales, the time interval of interest is $[10^{-6},10^{10}]$. Therefore, a constant time step size is not appropriate.
In the numerical simulations we use $\Delta t_i = 2^{i-1}\Delta t_0$ with $\Delta t_0=10^{-6}$ in the $i$th time step.
The small initial time step size $\Delta t_0$ is chosen to obtain an adequate resolution of $y_2$.
\section{Numerical results}\label{sec:numres}
In this section, we confirm the theoretical convergence order of the MPRK schemes, that we introduced in the preceding sections. 
We compare MPRK22 to MPRK22ncs schemes and investigate the influence of the parameter $\alpha$ on the truncation error of these schemes.
We also show approximations of MPRK22 and MPRK22ncs schemes applied to the stiff Robertson problem. 

To visualize the order of the MPRK schemes we use a relative error $E$ taken over all time steps and all constituents:
\begin{equation*}
 E = \frac{1}{N}\sum_{i=1}^N E_i,\quad E_i = \Bigl(\frac{1}{M}\sum_{m=1}^M y_i(t^m)\Bigr)^{\!\!-1}\Bigl(\frac{1}{M}\sum_{m=1}^M\left(y_i(t^m) - y_i^m\right)^2\Bigr)^{\!\!1/2},
\end{equation*}
where $M$ denotes the number of executed time steps. 
To compute the error $E$ we need to know the analytic solution, which is known for the linear test case, but not for the other test problems.
Hence, we computed a reference solution, using the \textsc{Matlab} functions \texttt{ode45} for the non-stiff problems and \texttt{ode23s} for the Robertson problem. 
In both cases we utilized the tolerances $\mathtt{AbsTol}=\mathtt{RelTol}=10^{-10}$.
\subsection*{Convergence order}
\begin{figure}[htb]
\begin{subfigure}{.49\textwidth}
\centering
\includegraphics[width=\textwidth]{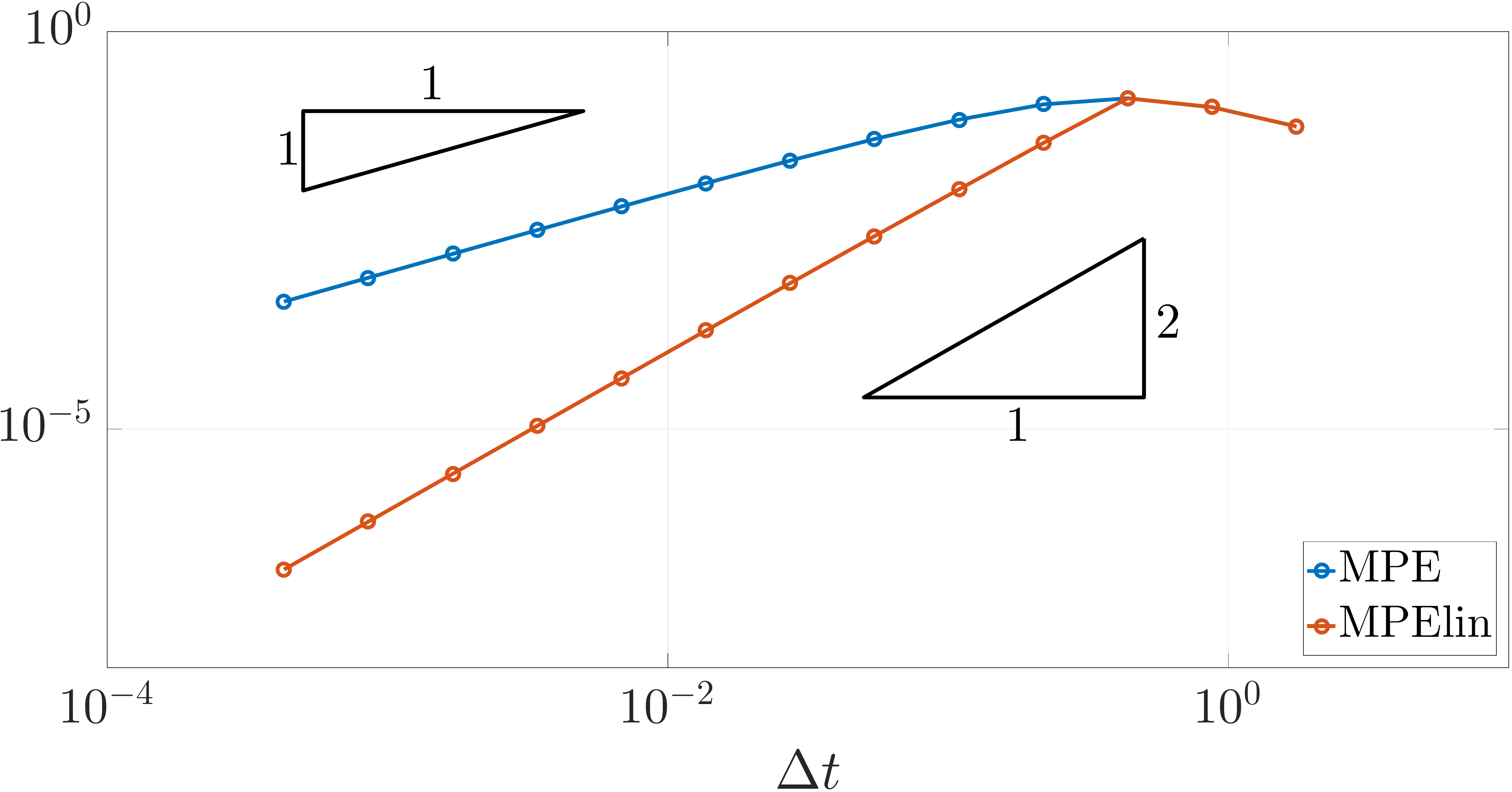}
\caption{MPE and MPElin.}
\label{fig:testlinorder1}
\end{subfigure}
\hfill
\begin{subfigure}{.49\textwidth}
\centering
\includegraphics[width=\textwidth]{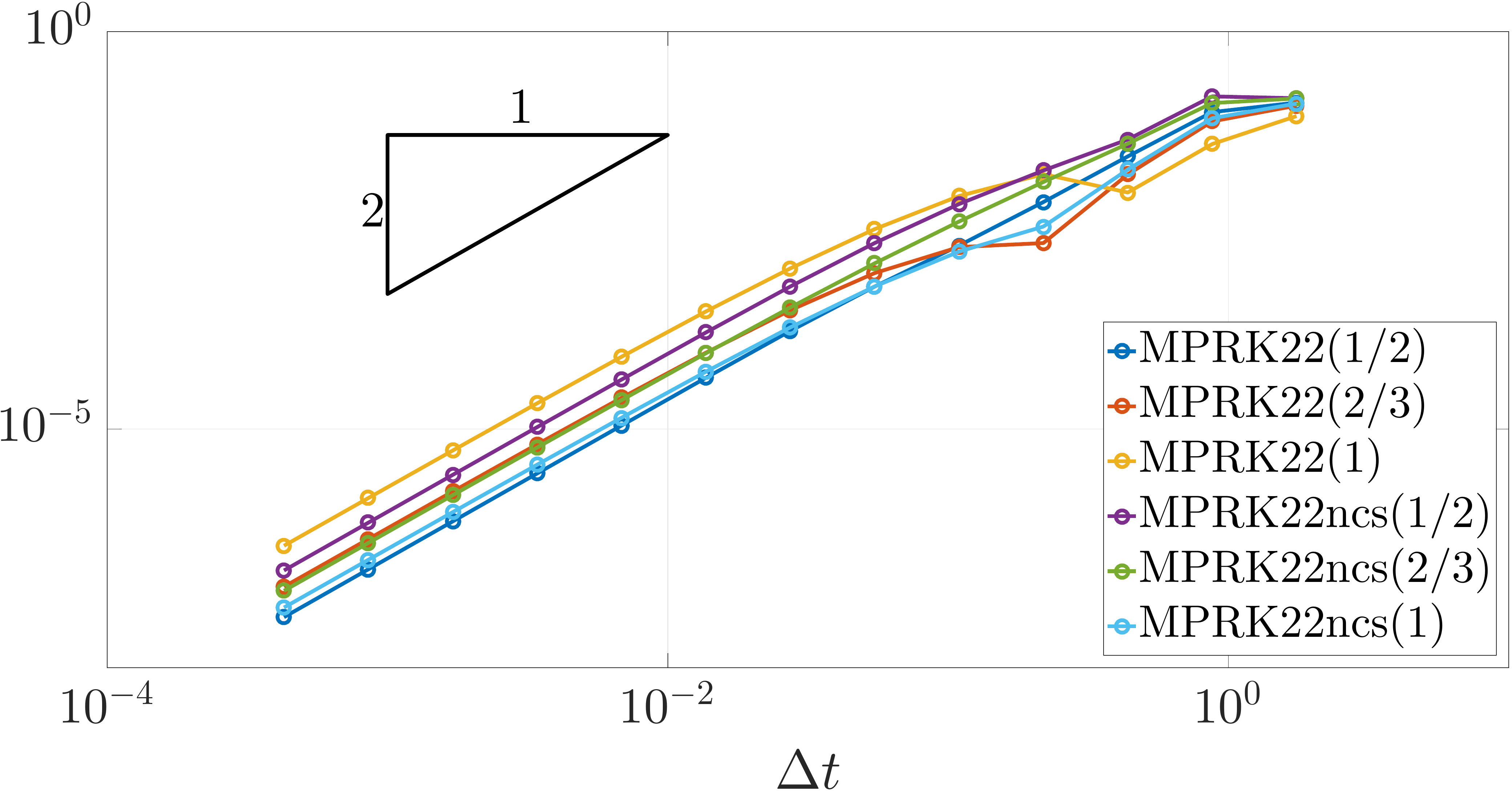}
\caption{MPRK22 and MPRK22ncs schemes.}
\label{fig:testlinorder2}
\end{subfigure}
\caption{Error plots of MPRK schemes solving the linear test problem \eqref{eq:lintest}.}
\label{fig:testlinorder}
\end{figure}

\begin{figure}[hbt]
\begin{subfigure}{.49\textwidth}
\centering
\includegraphics[width=\textwidth]{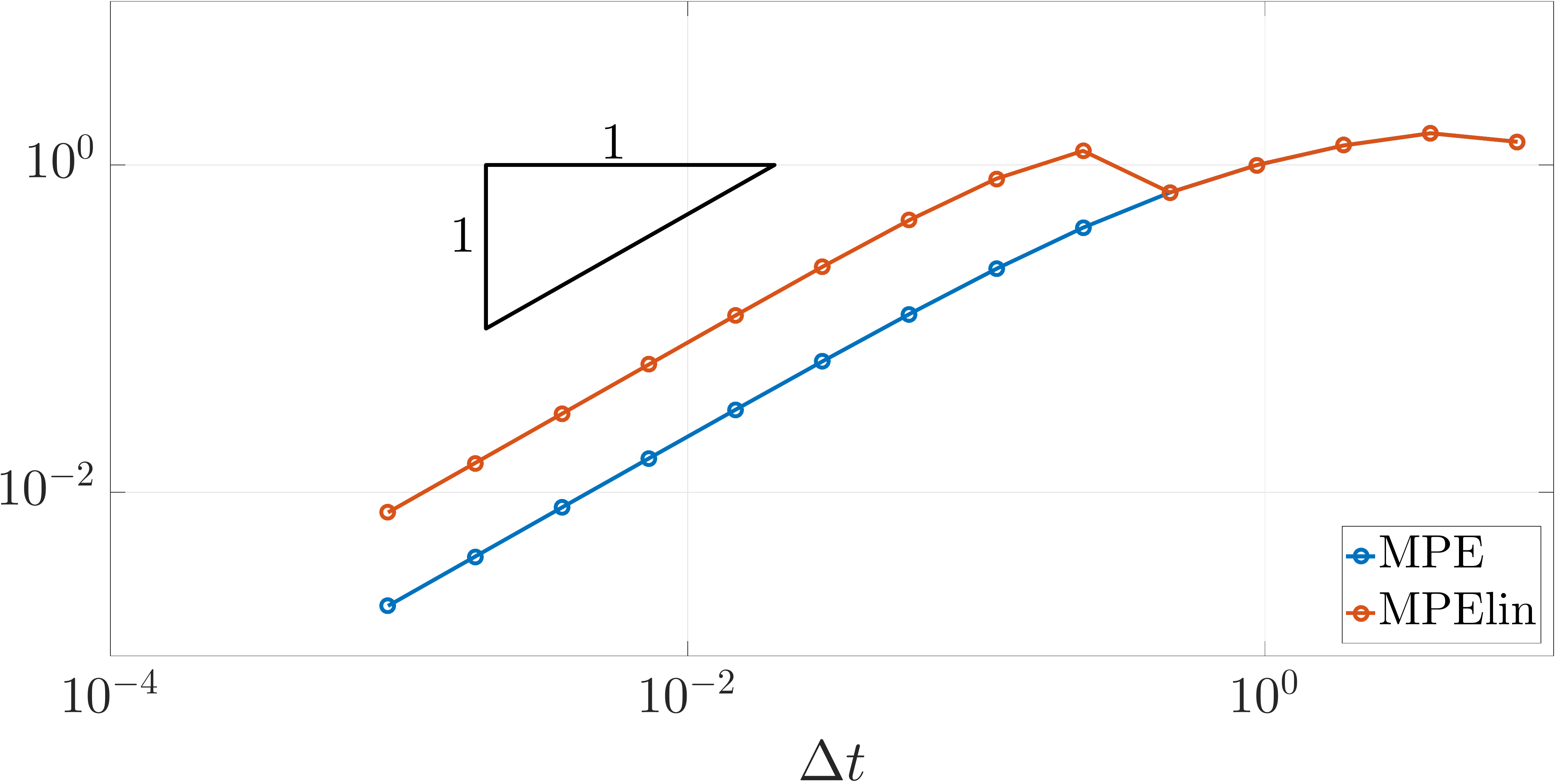}
\caption{MPE and MPElin.}
\label{fig:testnonlinorder1}
\end{subfigure}
\hfill
\begin{subfigure}{.49\textwidth}
\centering
\includegraphics[width=\textwidth]{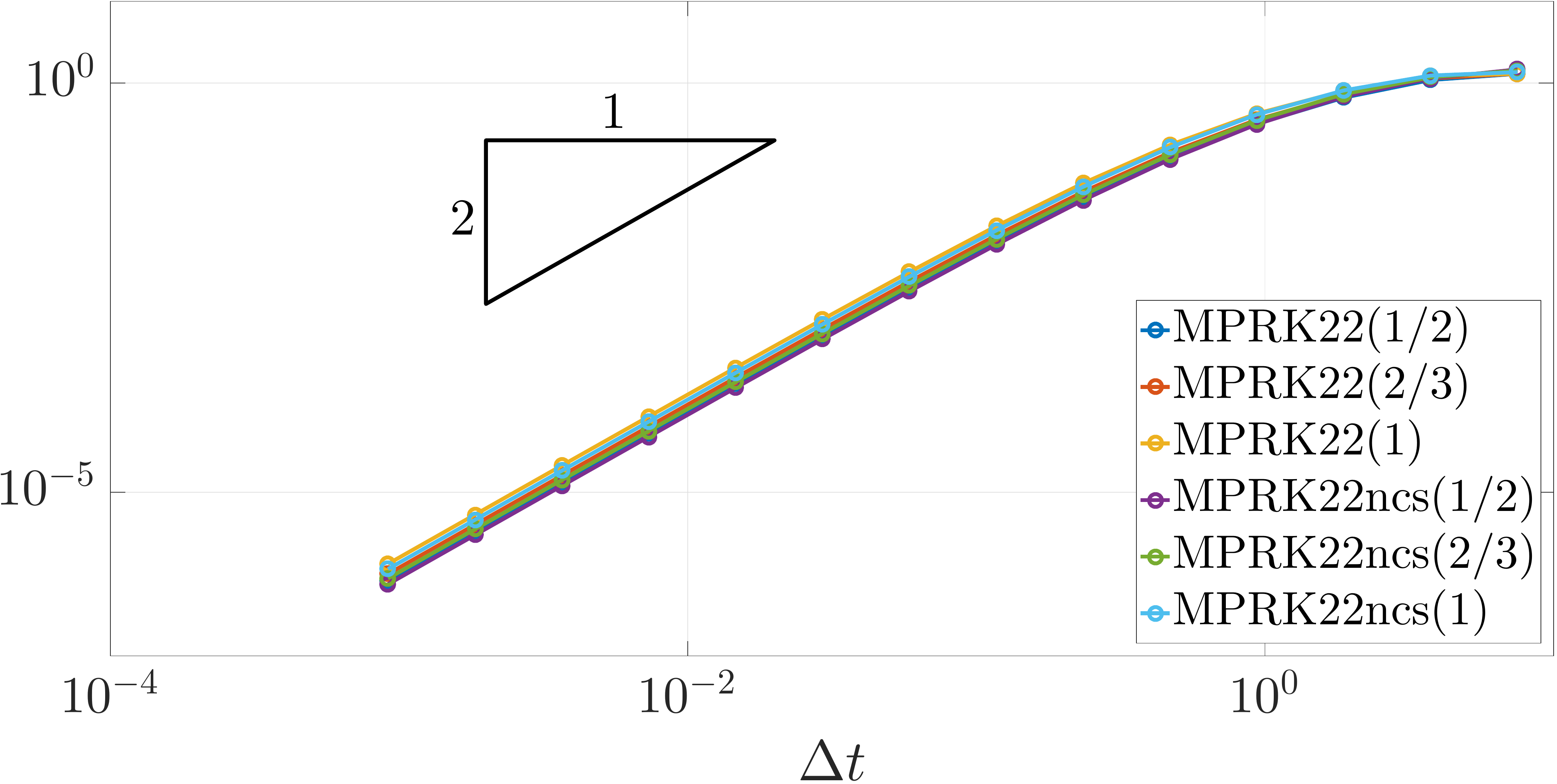}
\caption{MPRK22 and MPRK22ncs schemes.}
\label{fig:testnonlinorder2}
\end{subfigure}
\caption{Error plots of MPRK schemes solving the nonlinear test problem \eqref{eq:nonlintest}.}
\label{fig:testnonlinorder}
\end{figure}
Figure~\ref{fig:testlinorder} shows error plots of eight MPRK schemes applied to the linear test problem \eqref{eq:lintest}.
Figure~\ref{fig:testlinorder1} confirms that the MPE method \eqref{eq:MPE} is first order accurate and that the MPElin scheme \eqref{eq:MPRK11}, \eqref{eq:sigMPElin}, which was designed to be of second order, when applied to the linear test problem, shows the expected order of accuracy.
Owing to \eqref{eq:sigMPElin}, MPElin and MPE generate equal approximations as long as $\Delta t\geq 1/3$.
Figure~\ref{fig:testlinorder2} verifies the second order accuracy of MPRK22($\alpha$) and MPRK22ncs($\alpha$) for $\alpha\in\{1/2,2/3,1\}$.
These are the MPRK schemes corresponding to Heun's method ($\alpha=1$), the midpoint method ($\alpha=1/2$) and Ralston's method ($\alpha=2/3$). 
In addition, Figure~\ref{fig:testnonlinorder} shows error plots of the same schemes, when applied to the nonlinear test problem \eqref{eq:nonlintest}.
Again, we find the second order convergence of the MPRK22 and MPRK22ncs schemes, as well as the first order convergence of the MPE scheme. 
When applied to a problem other than \eqref{eq:lintest}, MPElin is only a first order scheme, which becomes evident in Figure~\ref{fig:testnonlinorder1}.

\subsection*{Truncation error}
Figure \ref{fig:testlinorder2} enables a comparison of  MPRK22 and MPRK22ncs for a fixed value of $\alpha$. 
One might expect MPRK22ncs($\alpha$) to be more accurate than MPRK22($\alpha$), since less weighting disturbs the original Runge-Kutta scheme.
But we see that MPRK22(1) is less accurate than MPRK22ncs(1) and MPRK22(1/2) is more accurate than MPRK22ncs(1/2) in the case of the linear test problem. 
Hence, one cannot make a general statement, if MPRK22($\alpha$) or MPRK22ncs($\alpha$) is more accurate.

\begin{figure}
\centering
\begin{subfigure}{.49\textwidth}
 \includegraphics[width=\textwidth]{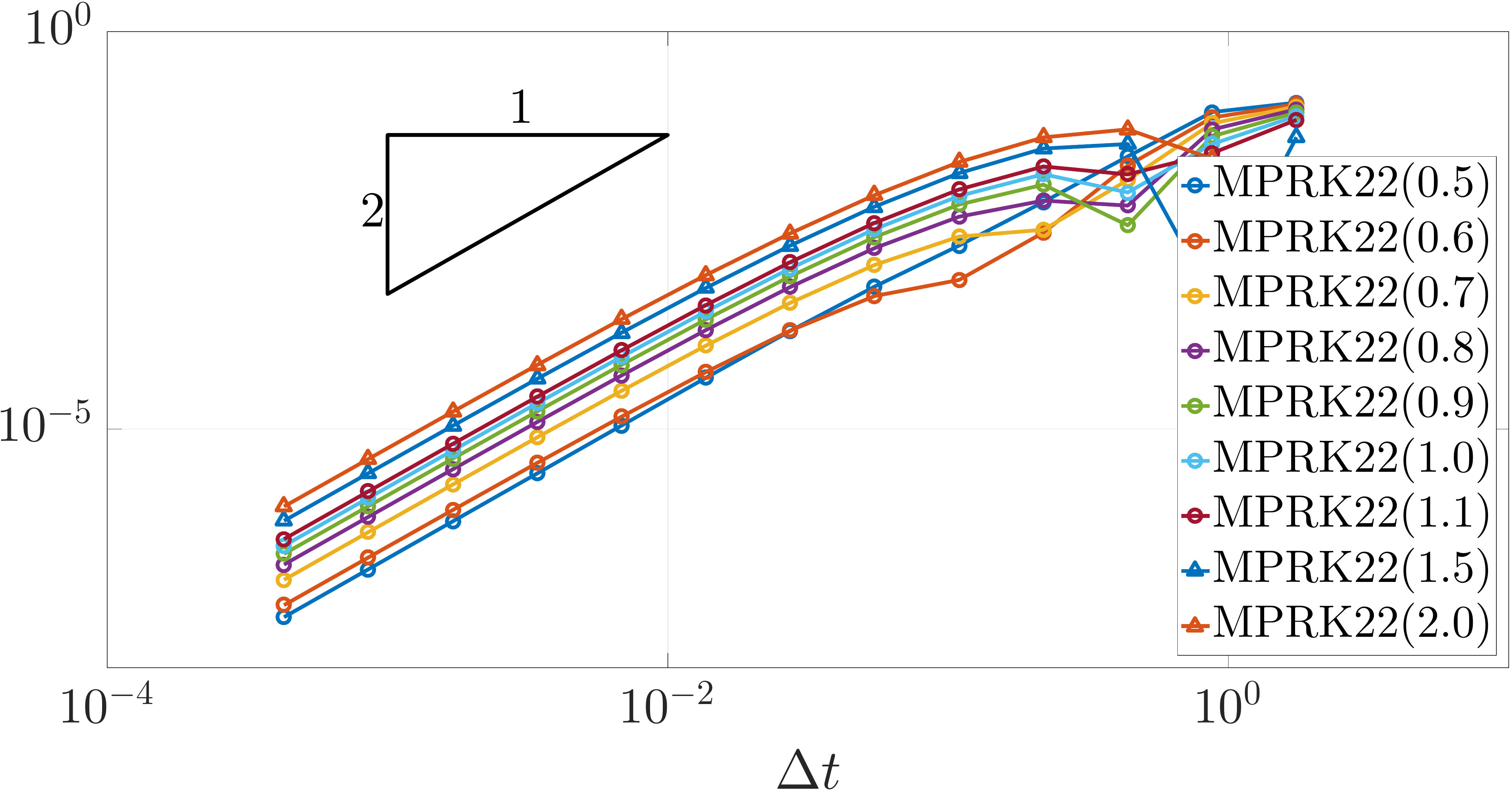}
 \caption{Linear test problem \eqref{eq:lintest}.}
\end{subfigure}
\hfill
\begin{subfigure}{.49\textwidth}
  \includegraphics[width=\textwidth]{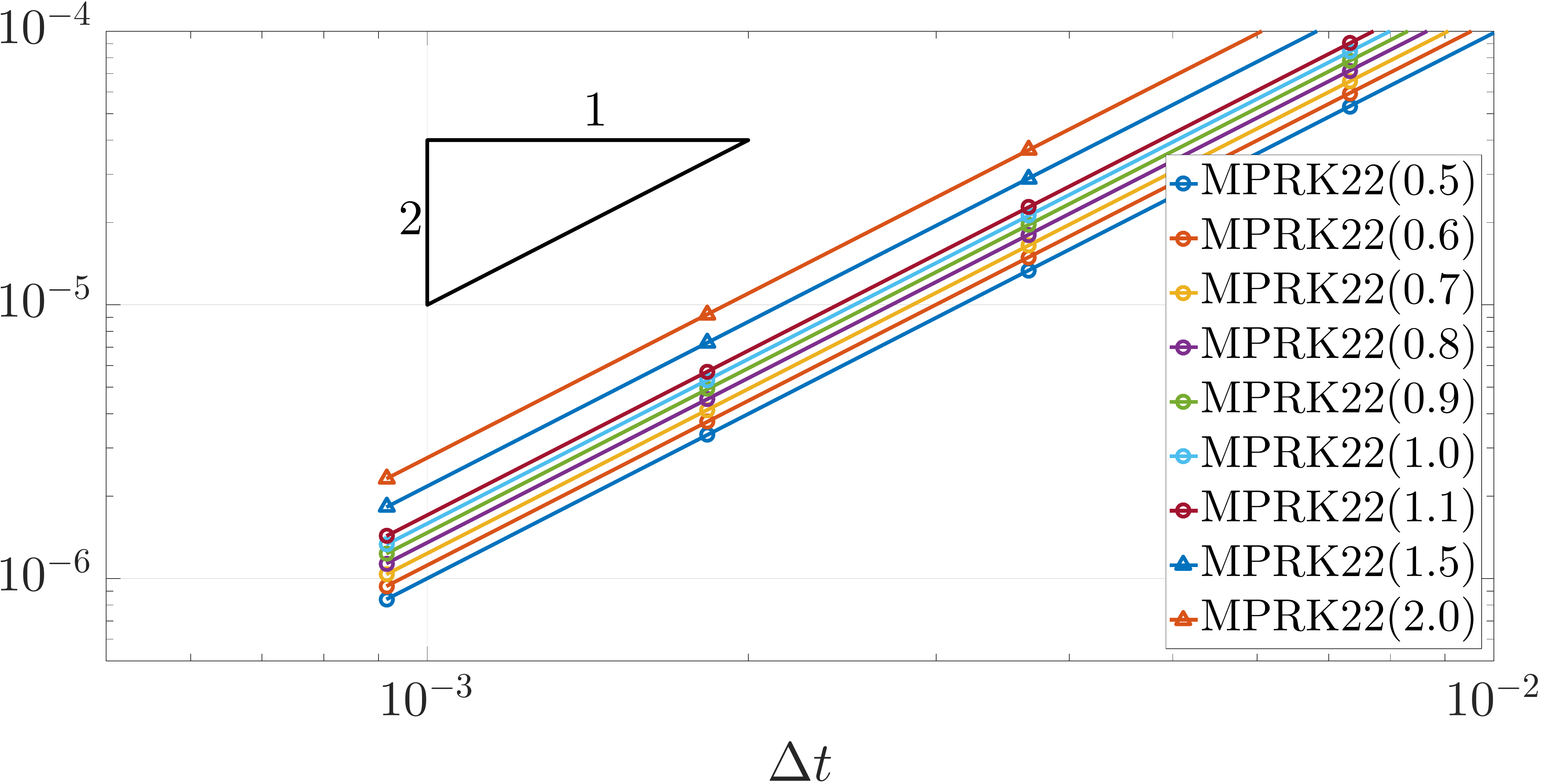}
  \caption{Nonlinear test problem \eqref{eq:nonlintest}.}
\end{subfigure}
\par
\vspace{2\baselineskip}
\begin{subfigure}{.49\textwidth}
  \includegraphics[width=\textwidth]{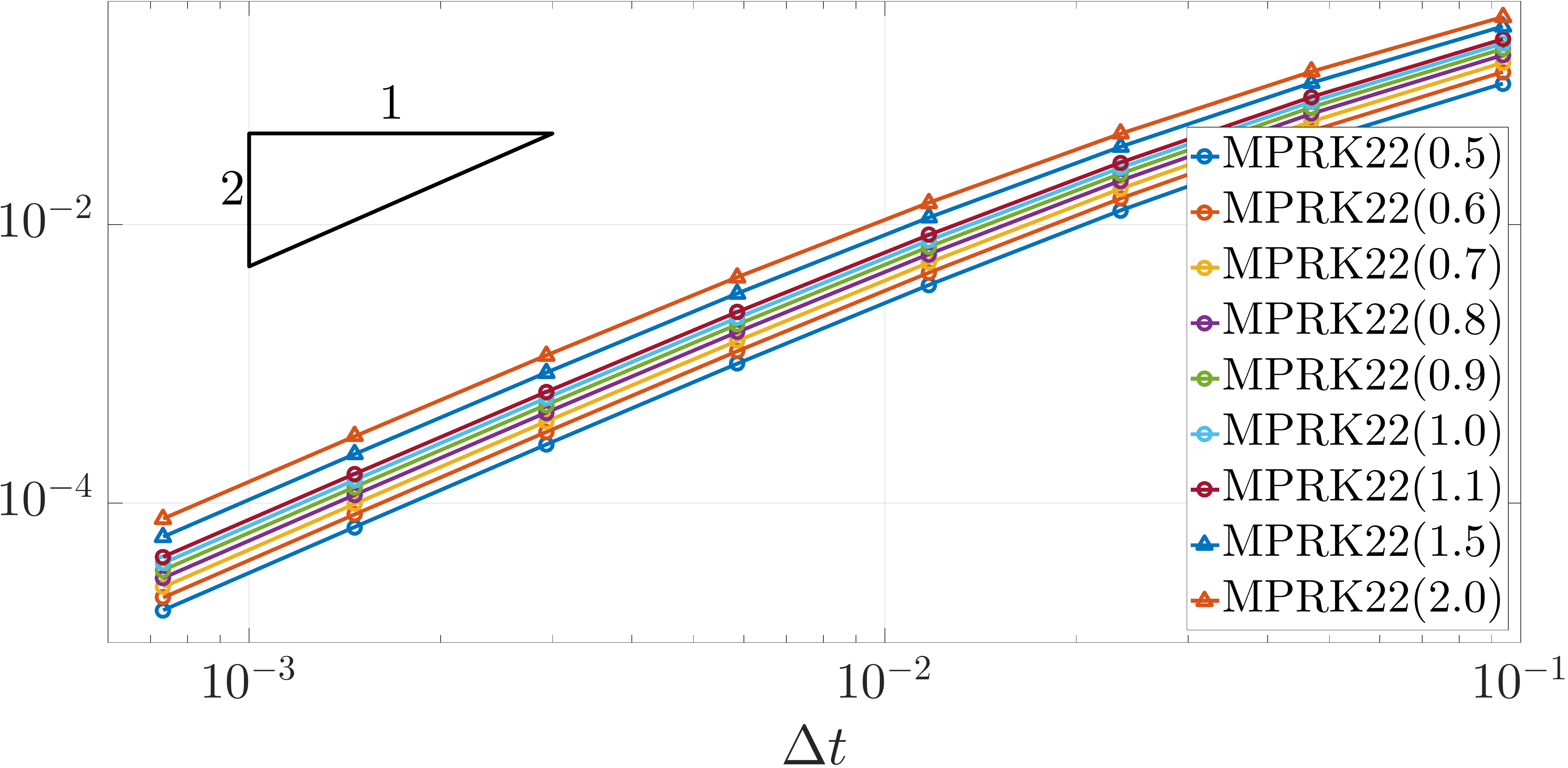}
  \caption{Brusselator problem \eqref{eq:brusselator}.}
\end{subfigure}
\caption{Error plots of MPRK22 schemes for various values of $\alpha$ applied to different test problems.}
\label{fig:ErrorMPRK22}
\end{figure}

Figure~\ref{fig:ErrorMPRK22} shows error plots of nine MPRK22 schemes applied to the linear test problem \eqref{eq:lintest}, the nonlinear test problem \eqref{eq:nonlintest} and the Brusselator \eqref{eq:brusselator}.
The parameter $\alpha$ takes the values $\alpha=0.5,0.6,\dots,1.0,1.5,2.0$.
In all three cases, we see that MPRK22(1/2) generates the most accurate approximations and that the error seems to increase monotonically with the value of $\alpha$. 
This property is not shared by the MPRK22ncs and the explicit Runge-Kutta schemes.
Therefore, an analytical investigation of the truncation errors of the MPRK22 schemes is of high interest, to reveal if this is merely coincidental, due to similar properties of the test problems or a general rule.

\subsection*{Stiff problems and stability}
Figure~\ref{fig:robertson} shows numerical approximations of eight MPRK22 and MPRK22ncs schemes applied to the stiff Robertson problem \eqref{eq:robertson}. As mentioned, the time step size in the $k$th time step was chosen as $\Delta t_k=2^{k-1}\Delta t_0$ with initial time step size $\Delta t_0=10^{-6}$.
Hence, only 55 time steps are necessary to traverse the time interval $[10^{-6},10^{10}]$.
The small initial time step was chosen to obtain an adequate resolution of the component $y_2$ in the starting phase. 
To visualize the evolution of $y_2$, it was multiplied by $10^4$.

The MPRK22ncs schemes fail to produce adequate approximations (right column), when $\alpha$ is close to $1/2$.
The oscillations become less, as the value of $\alpha$ increases, and for $\alpha=1$ no oscillations can be observed (Figure~\ref{fig:robertsonMPRK22ncs1}).
When applied to solve the nonlinear test problem \eqref{eq:nonlintest} or the Brusselator \eqref{eq:brusselator} no oscillations are visible, see Figures~\ref{fig:nonlin} and \ref{fig:brusselator}.

In absence of a stability analysis of MPRK schemes, we can only speculate what causes these oscillations.
Therefore, such a stability analysis is vitally important and will be a major research topic in the future.

Nevertheless, we can hardly distinguish the MPRK22 approximations from the reference solution (left column), which shows the excellent accuracy of MPRK22 schemes even in the case of a highly stiff problem.

\begin{figure}[htb]
\centering
\begin{subfigure}{.49\textwidth}
\centering
 \includegraphics[width=\textwidth]{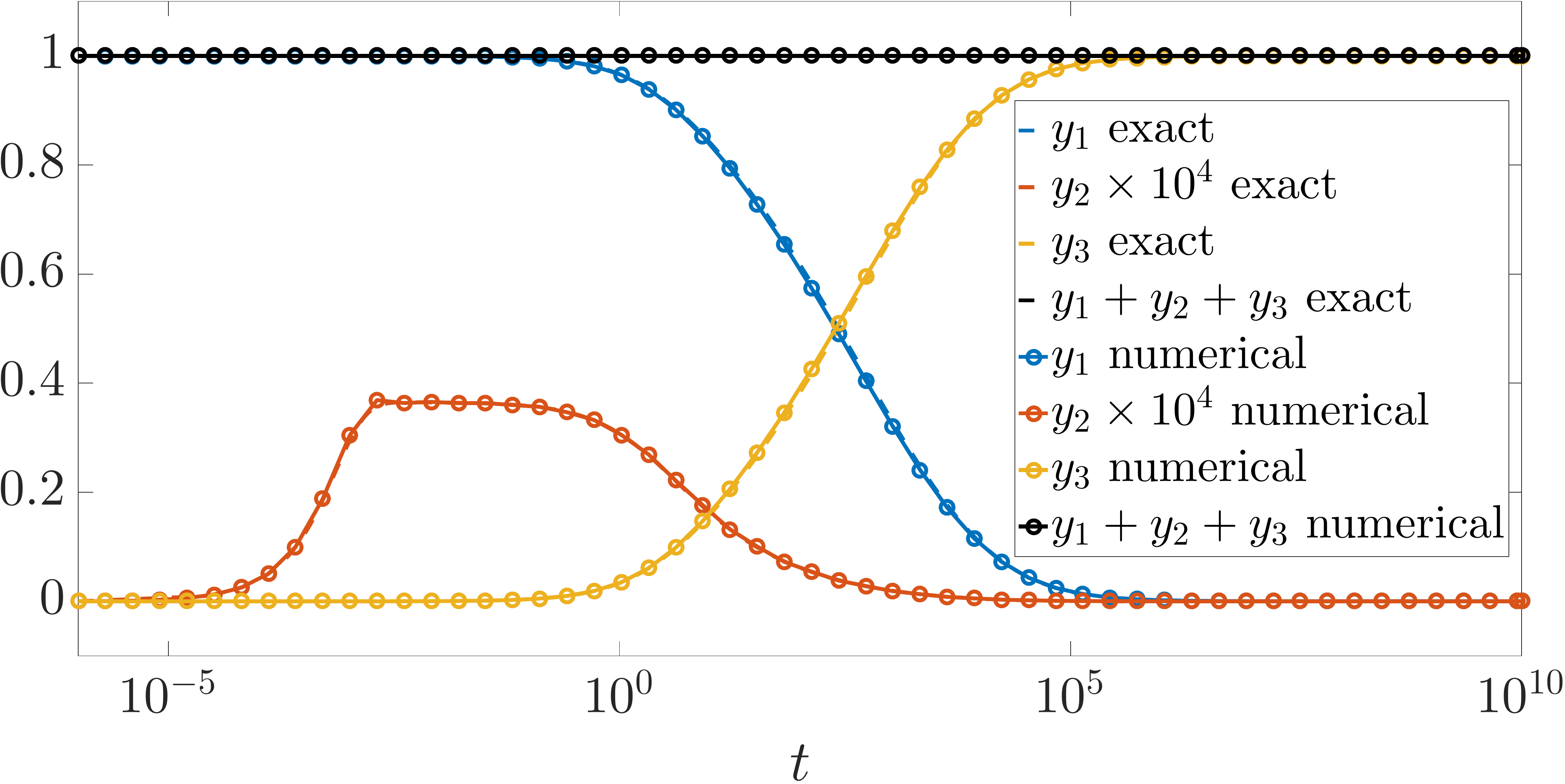}
 \caption{MPRK22($1/2$)}
 \end{subfigure}
 \hfill
\begin{subfigure}{.49\textwidth}
\centering
 \includegraphics[width=\textwidth]{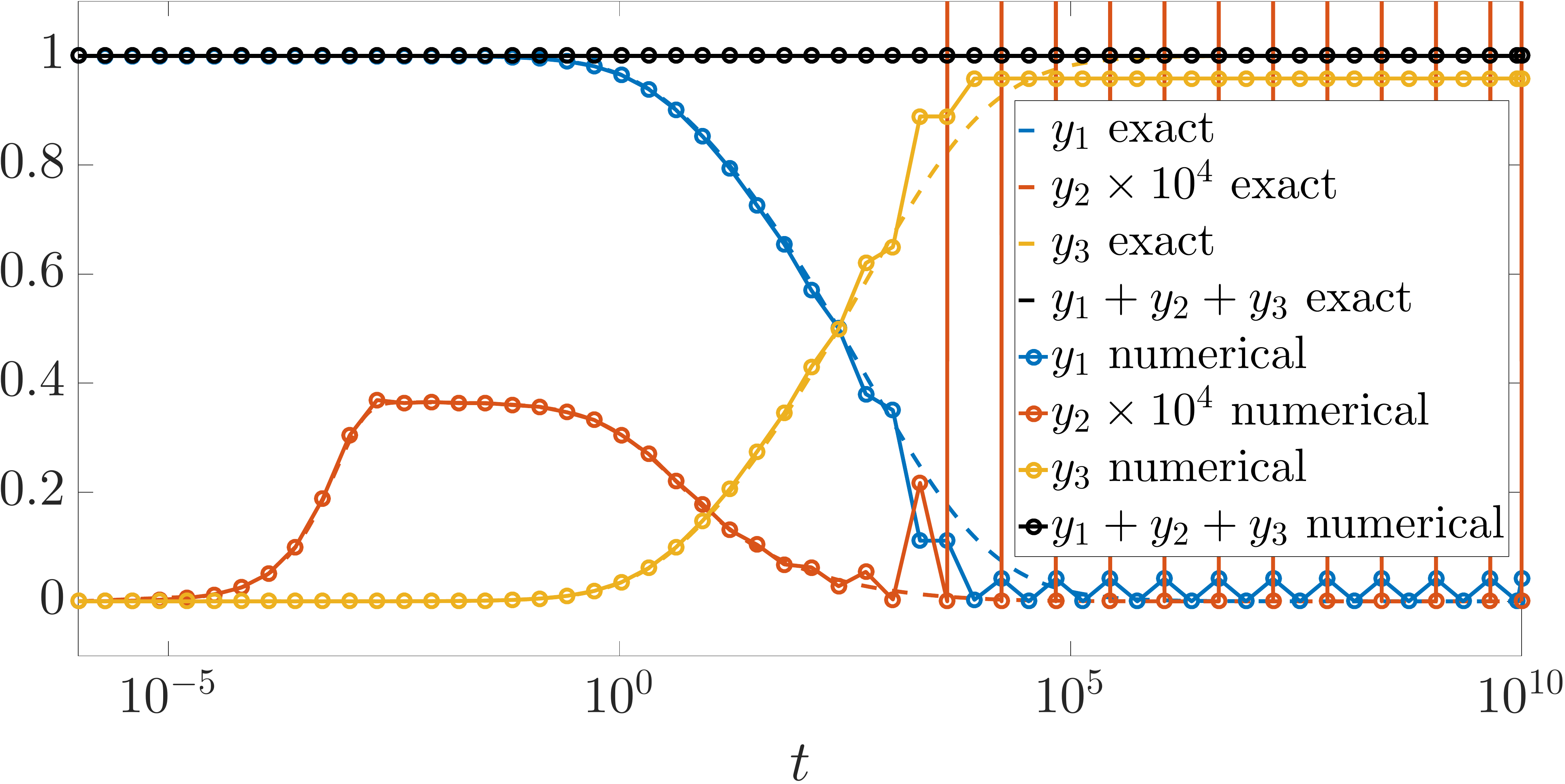}
 \caption{MPRK22ncs($1/2$)}
 \end{subfigure} 
 \par
\vspace{2\baselineskip} 
 \begin{subfigure}{.49\textwidth}
\centering
 \includegraphics[width=\textwidth]{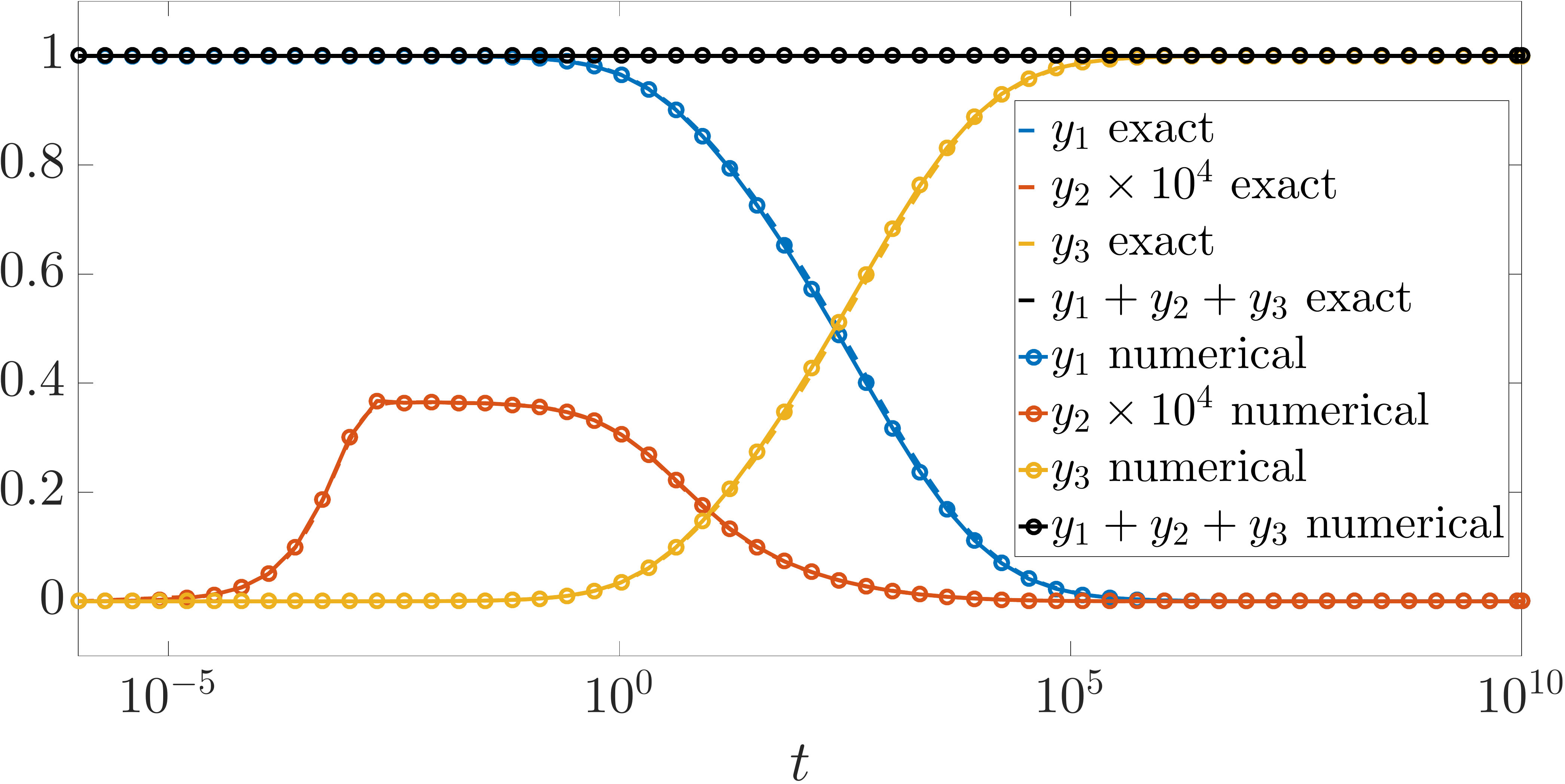}
 \caption{MPRK22($3/5$)}
 \end{subfigure}
\hfill
 \begin{subfigure}{.49\textwidth}
\centering
 \includegraphics[width=\textwidth]{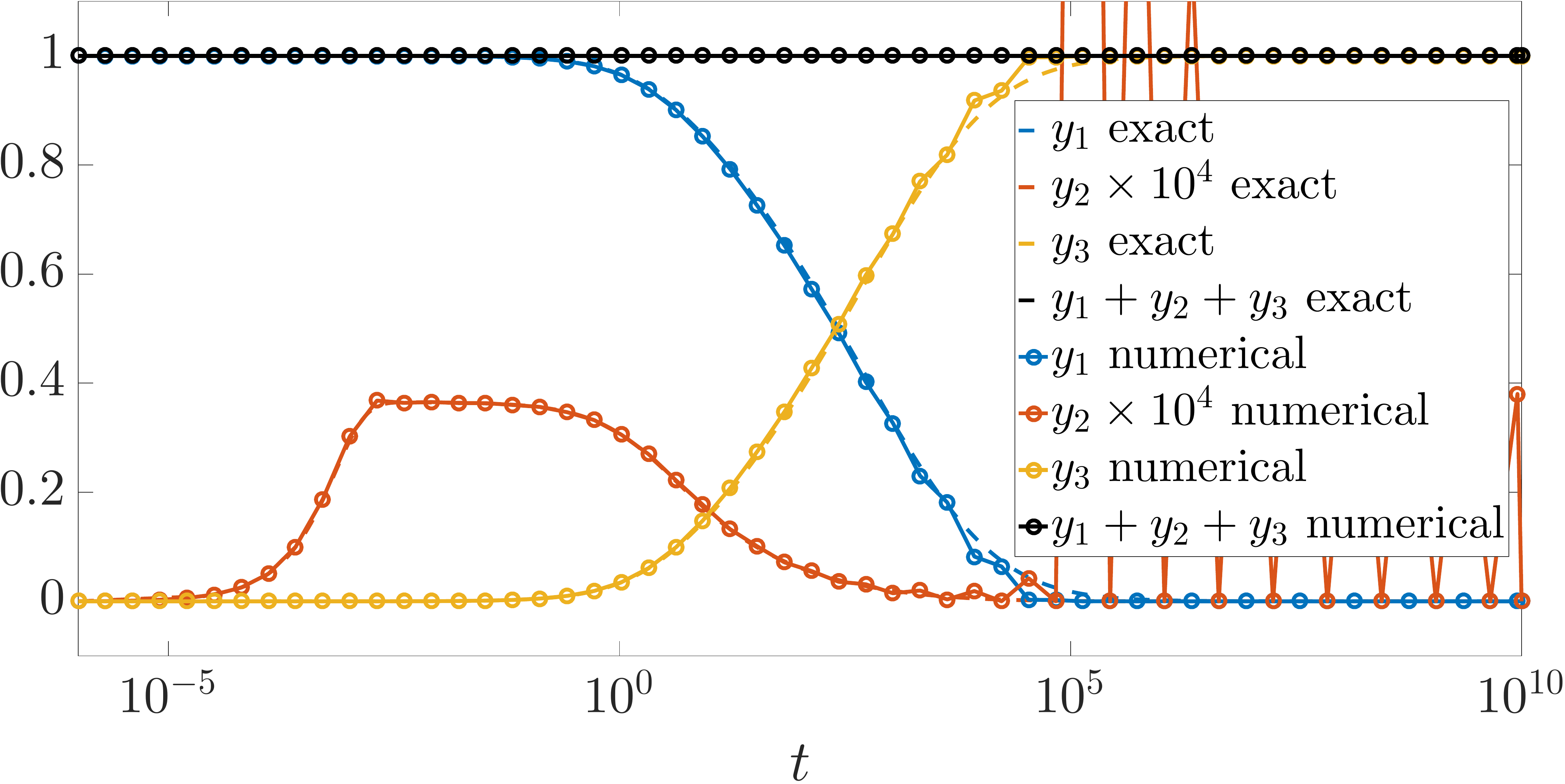}
 \caption{MPRK22ncs($3/5$)}
 \end{subfigure}
 \par\vspace{2\baselineskip} 
 \begin{subfigure}{.49\textwidth}
\centering
 \includegraphics[width=\textwidth]{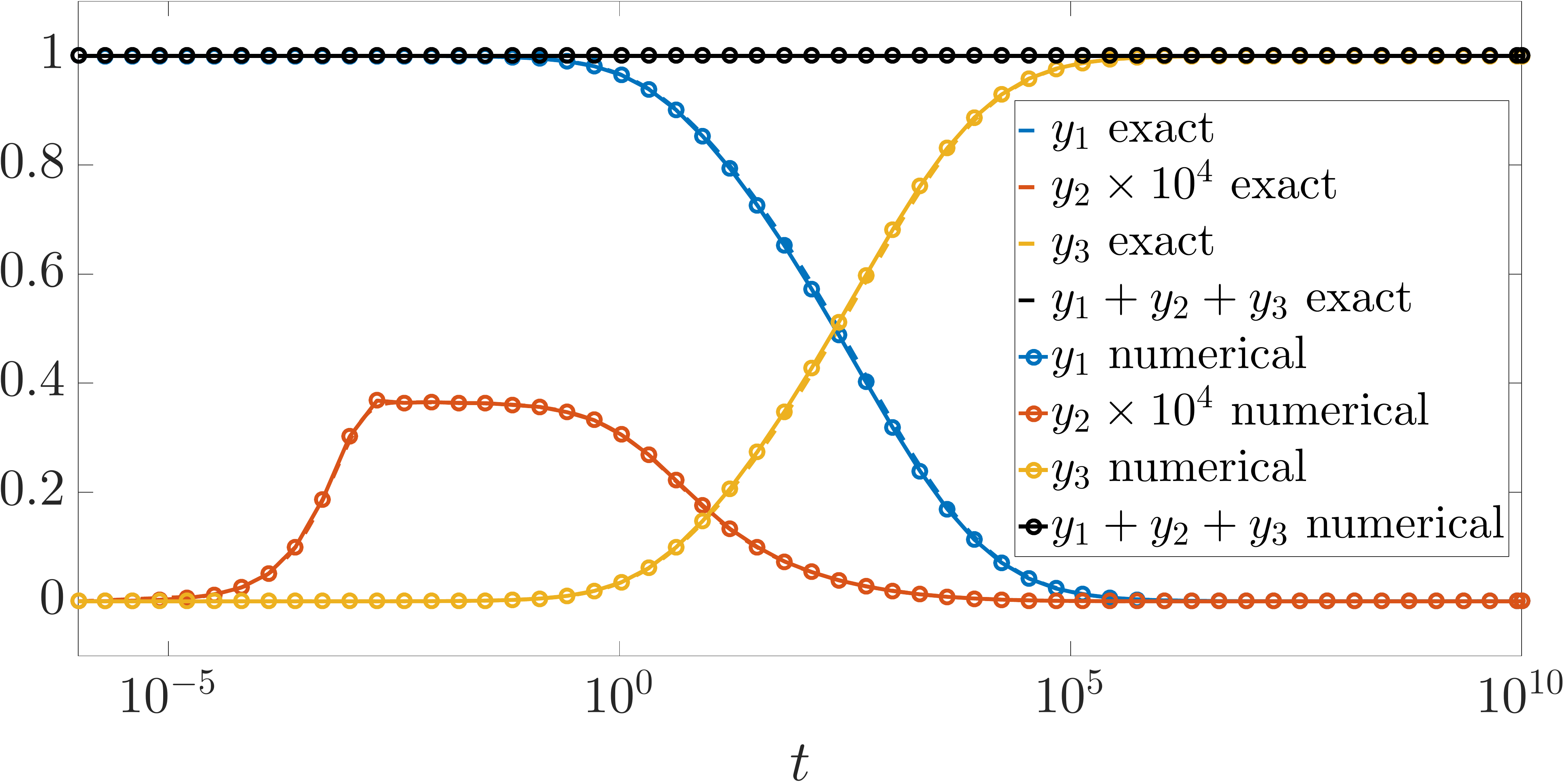}
 \caption{MPRK22($2/3$)}
 \end{subfigure}
\hfill
 \begin{subfigure}{.49\textwidth}
\centering
 \includegraphics[width=\textwidth]{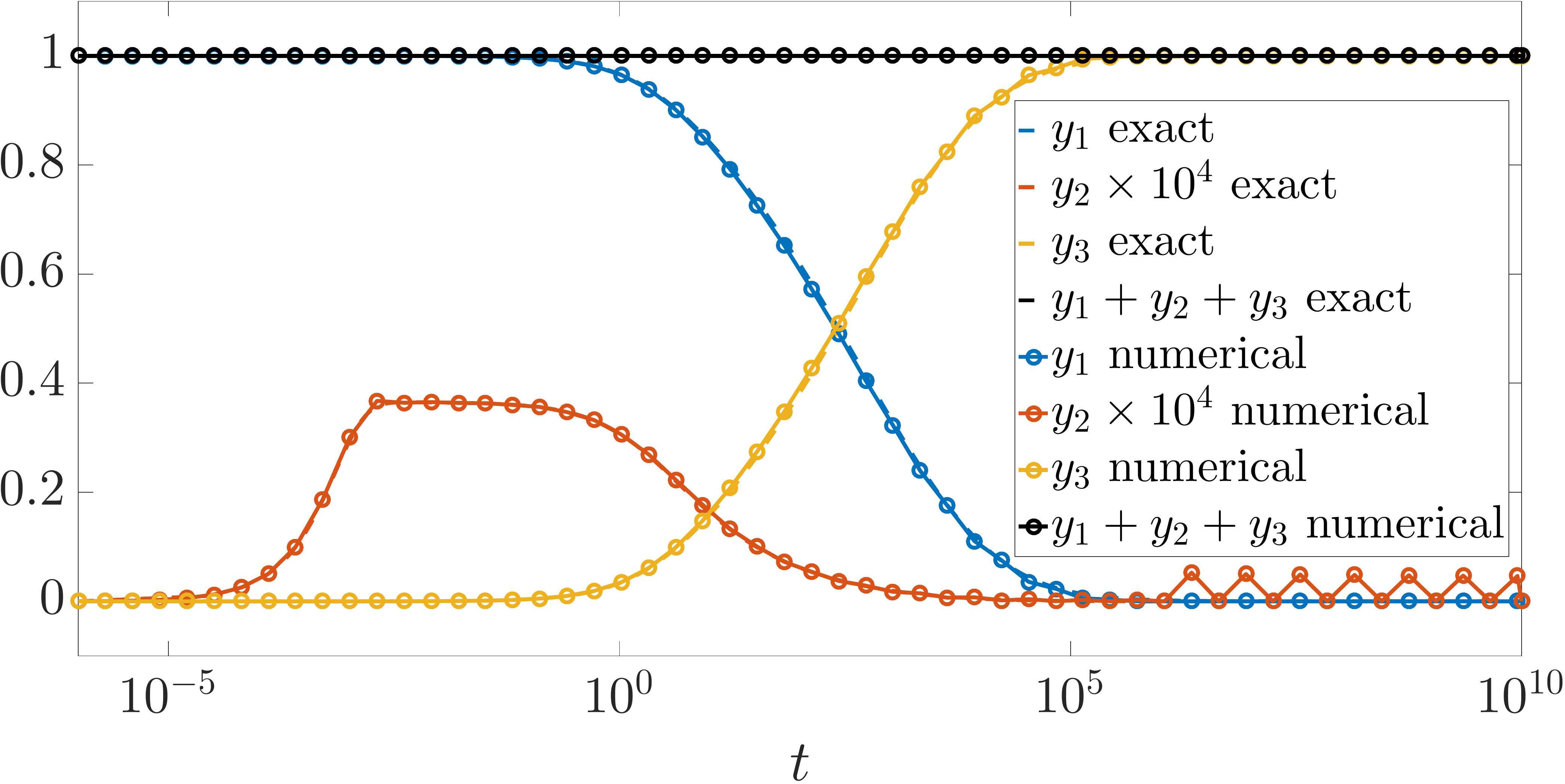}
 \caption{MPRK22ncs($2/3$)}
 \end{subfigure}
 \par
 \vspace{2\baselineskip}
  \begin{subfigure}{.49\textwidth}
\centering
 \includegraphics[width=\textwidth]{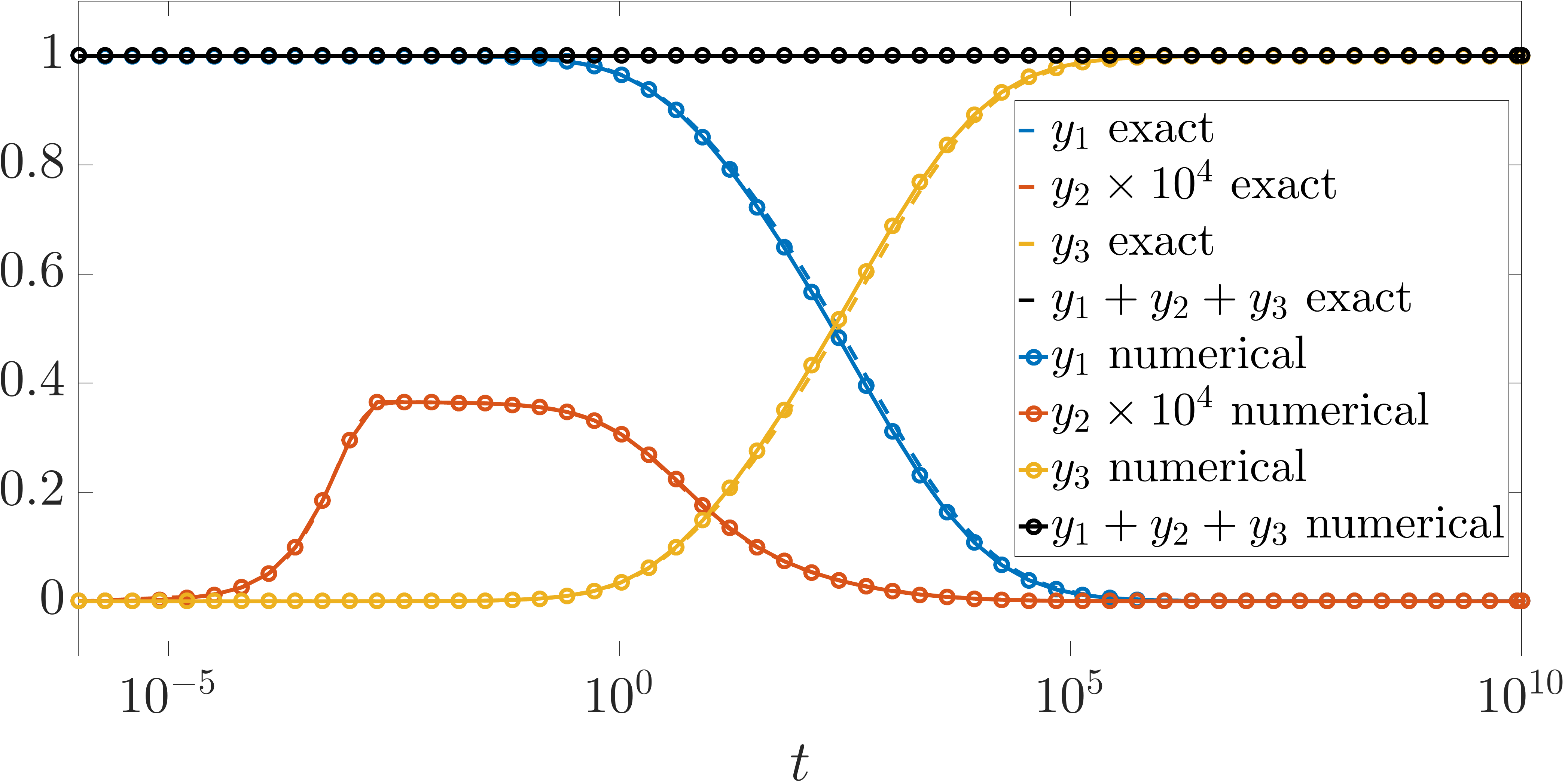}
 \caption{MPRK22(1)}
 \end{subfigure}
 \hfill
  \begin{subfigure}{.49\textwidth}
\centering
 \includegraphics[width=\textwidth]{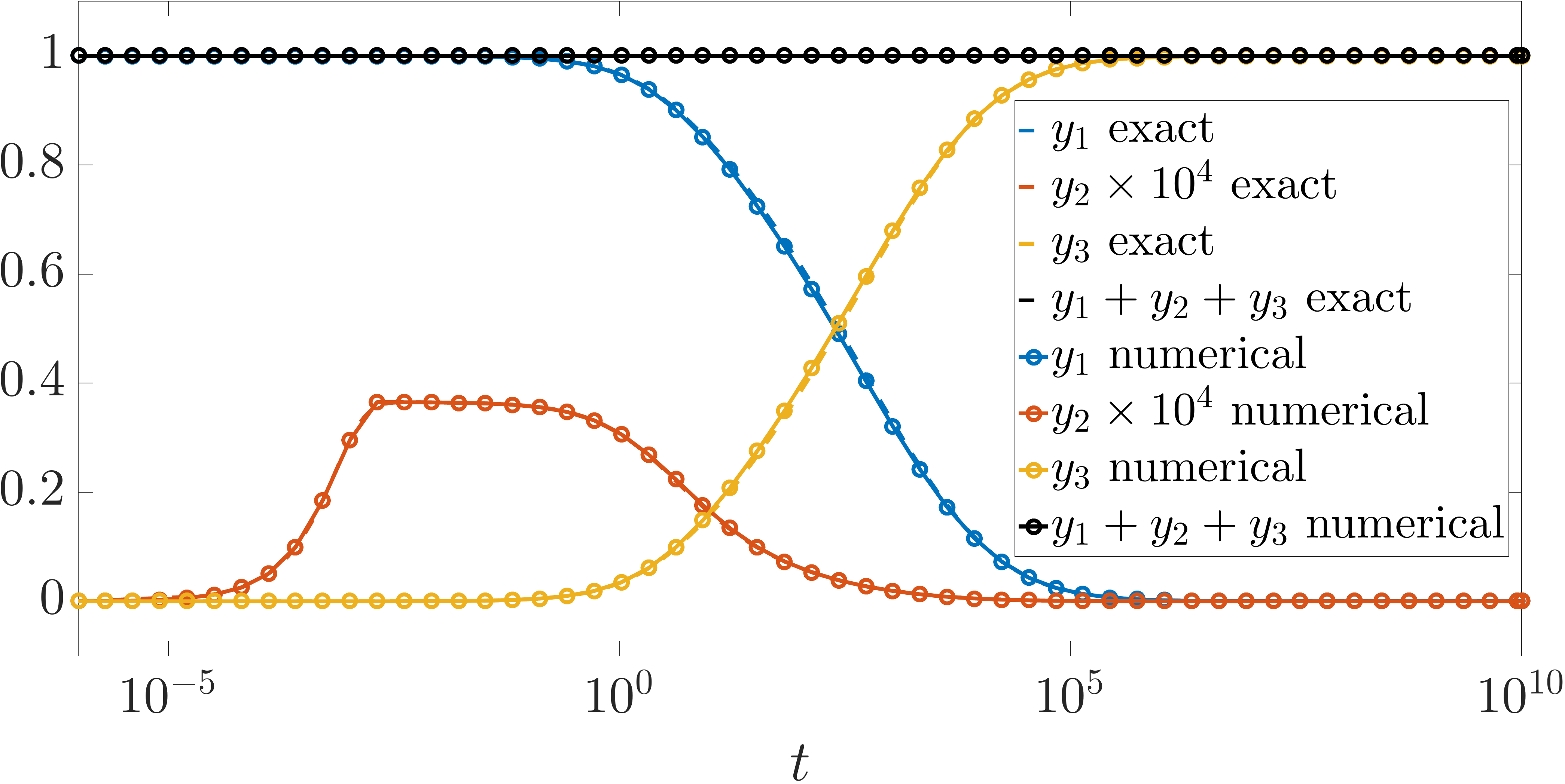}
 \caption{MPRK22ncs(1)}
 \label{fig:robertsonMPRK22ncs1}
 \end{subfigure}
 \caption{Numerical solutions of the Robertson problem \eqref{eq:robertson} for different MPRK22 and MPRK22ncs schemes.}
 \label{fig:robertson}
\end{figure}

\begin{figure}[htb]
\centering
\begin{subfigure}{.49\textwidth}
\centering
 \includegraphics[width=\textwidth]{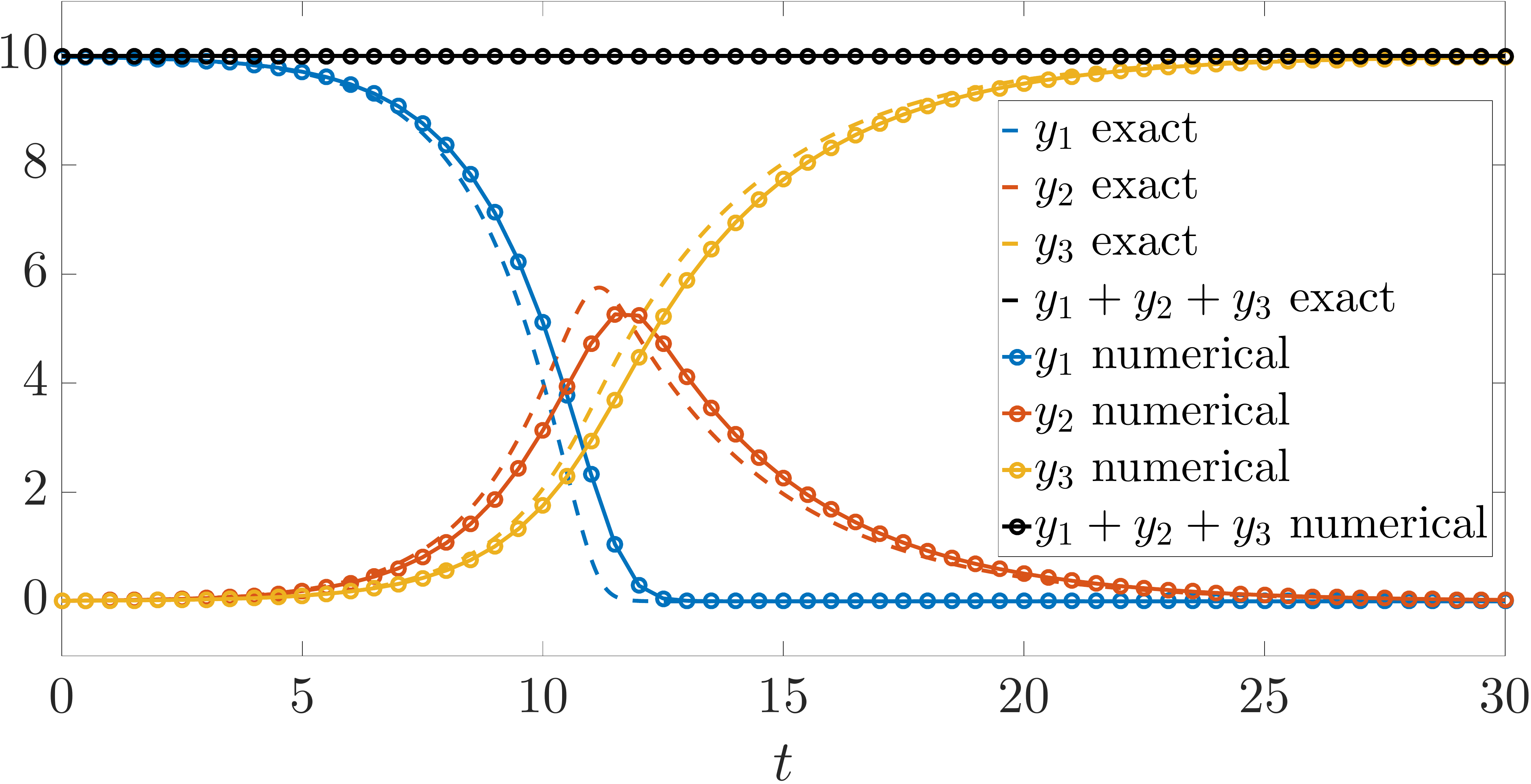}
 \caption{MPRK22($1/2$)}
 \end{subfigure}
 \hfill
\begin{subfigure}{.49\textwidth}
\centering
 \includegraphics[width=\textwidth]{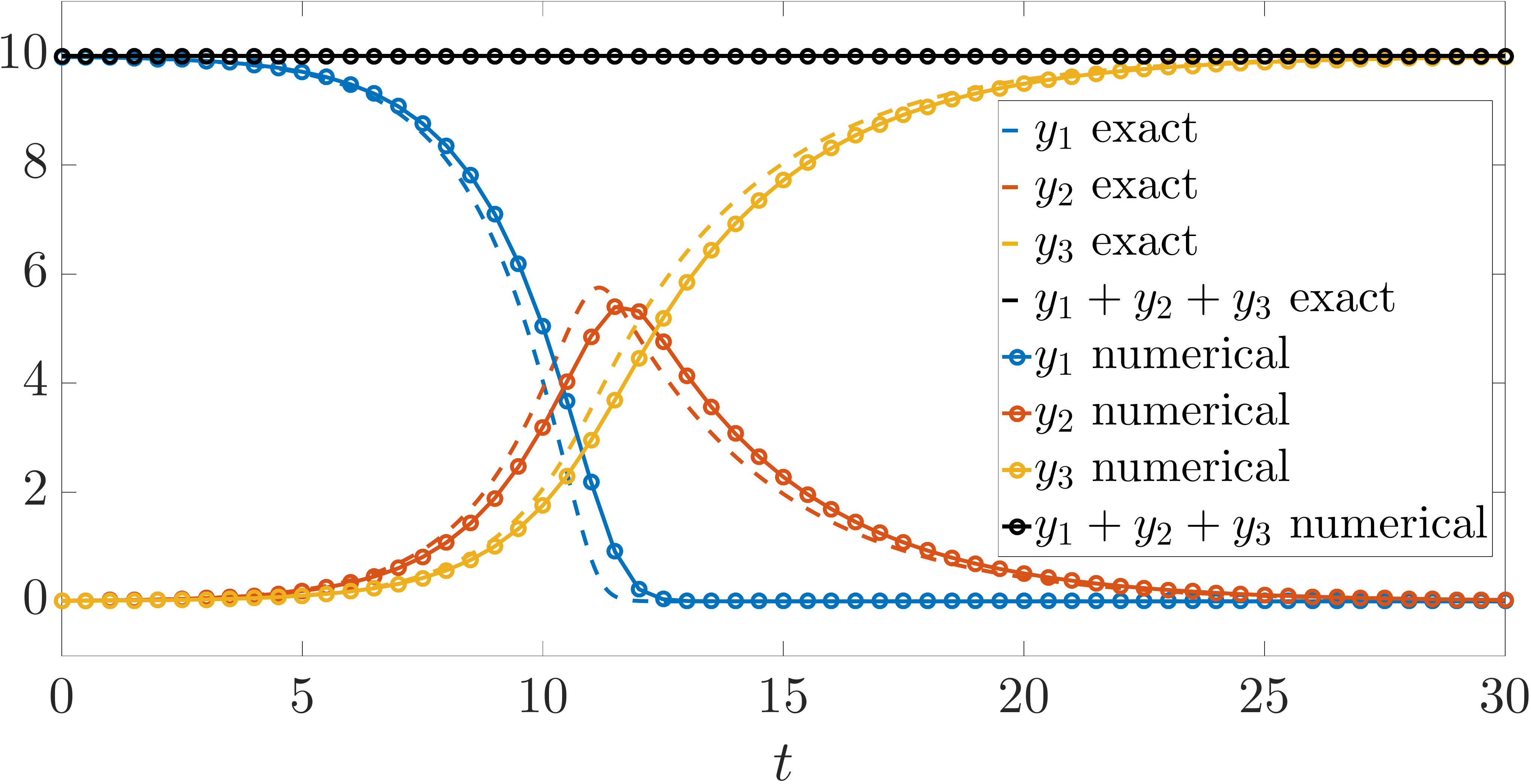}
 \caption{MPRK22ncs($1/2$)}
 \end{subfigure} 
 \par
\vspace{2\baselineskip} 
 \begin{subfigure}{.49\textwidth}
\centering
 \includegraphics[width=\textwidth]{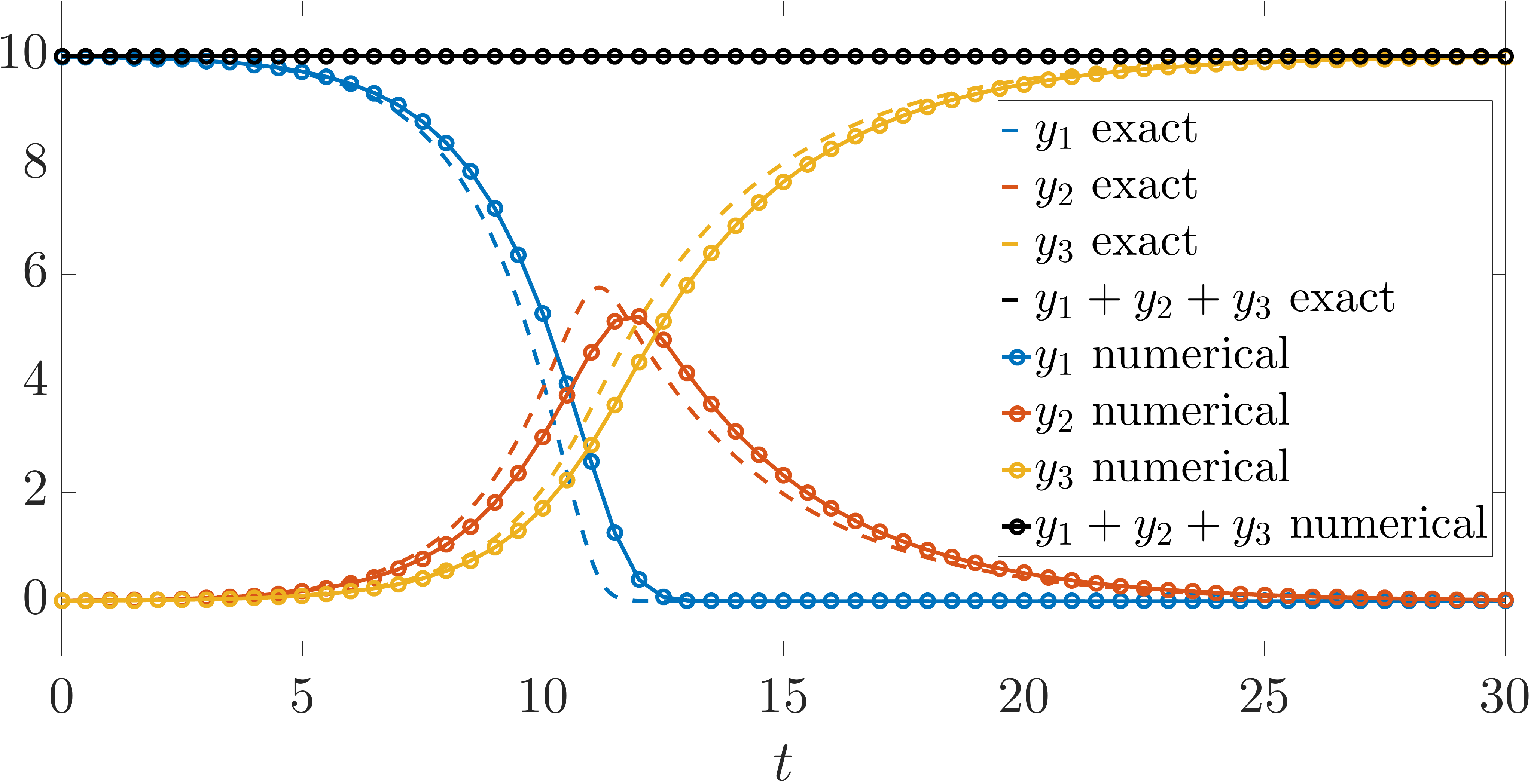}
 \caption{MPRK22($2/3$)}
 \end{subfigure}
\hfill
 \begin{subfigure}{.49\textwidth}
\centering
 \includegraphics[width=\textwidth]{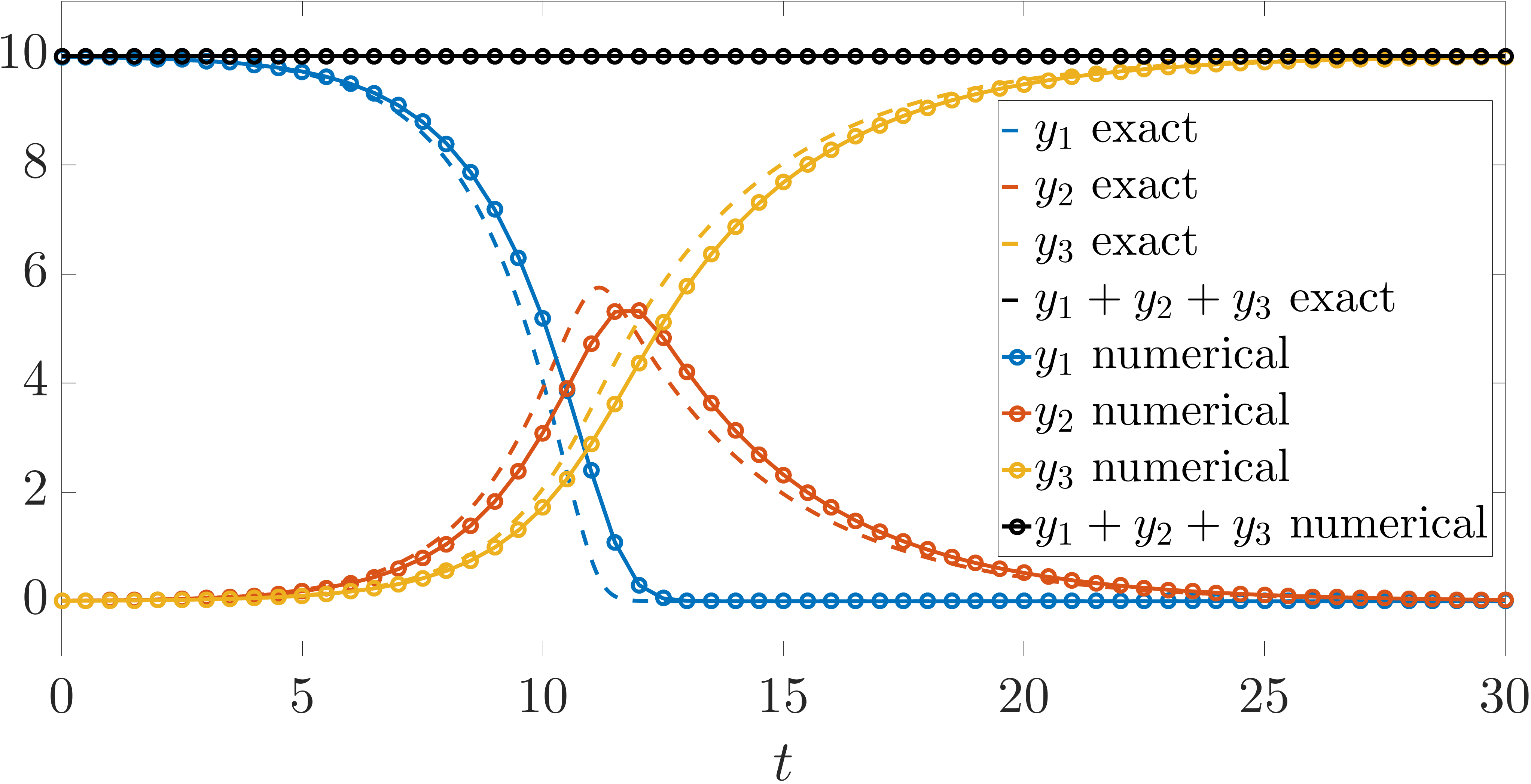}
 \caption{MPRK22ncs($2/3$)}
 \end{subfigure}
  \par
 \vspace{2\baselineskip}
  \begin{subfigure}{.49\textwidth}
\centering
 \includegraphics[width=\textwidth]{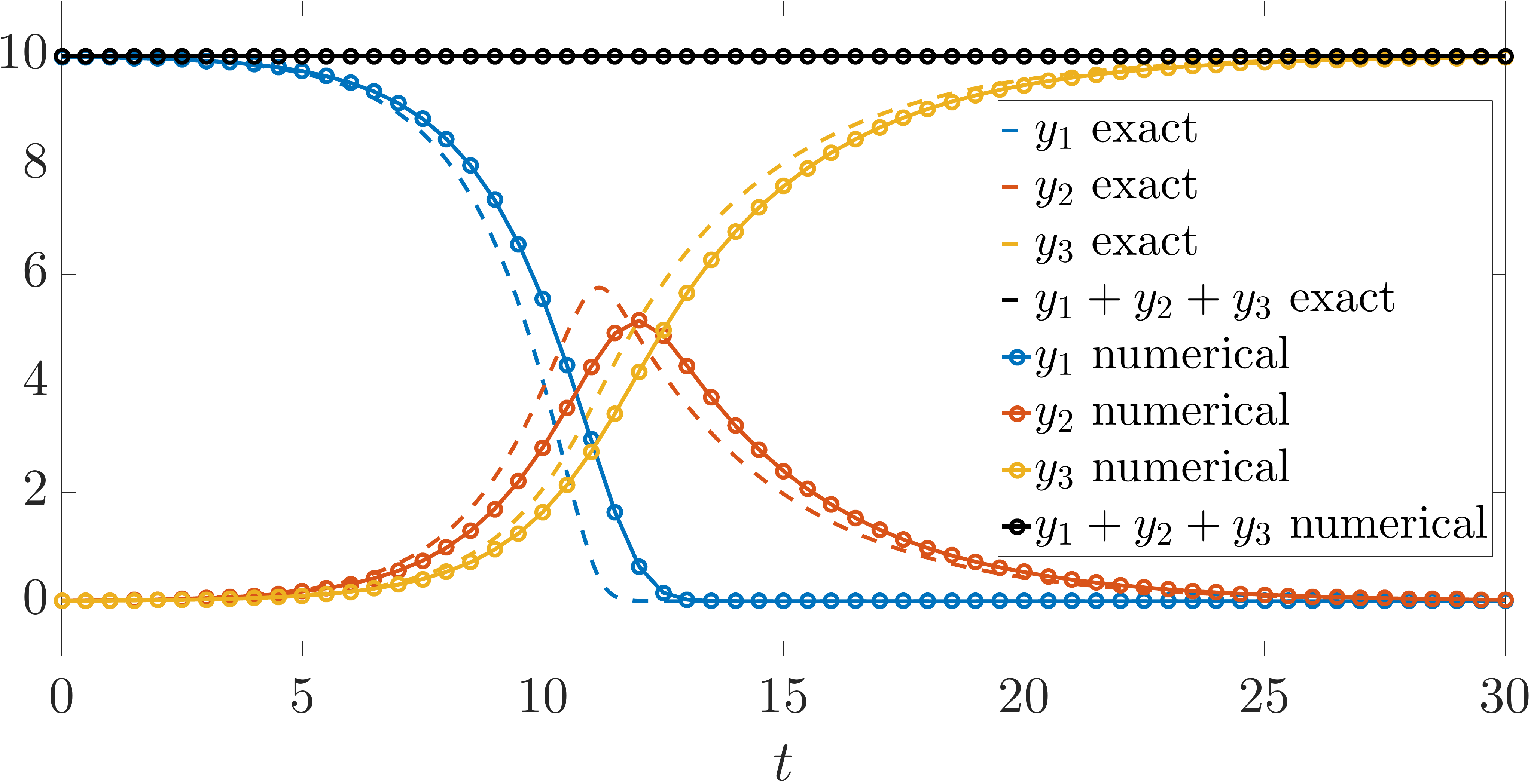}
 \caption{MPRK22(1)}
 \end{subfigure}
 \hfill
  \begin{subfigure}{.49\textwidth}
\centering
 \includegraphics[width=\textwidth]{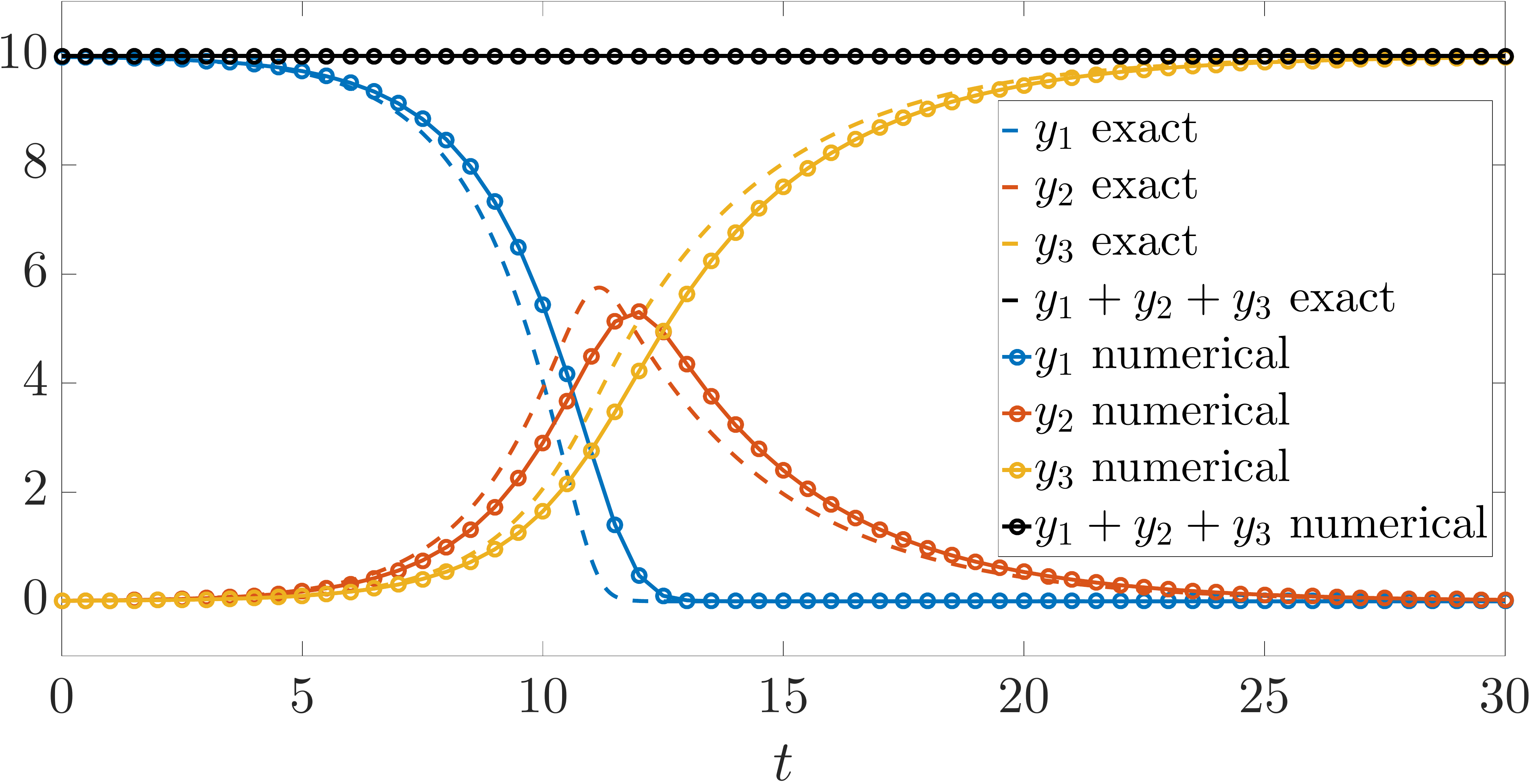}
 \caption{MPRK22ncs(1)}
 \end{subfigure}
 \caption{Numerical solutions of the nonlinear test problem \eqref{eq:nonlintest} for different MPRK22 and MPRK22ncs schemes.}
 \label{fig:nonlin}
\end{figure}

\begin{figure}[htb]
\centering
\begin{subfigure}{.49\textwidth}
\centering
 \includegraphics[width=\textwidth]{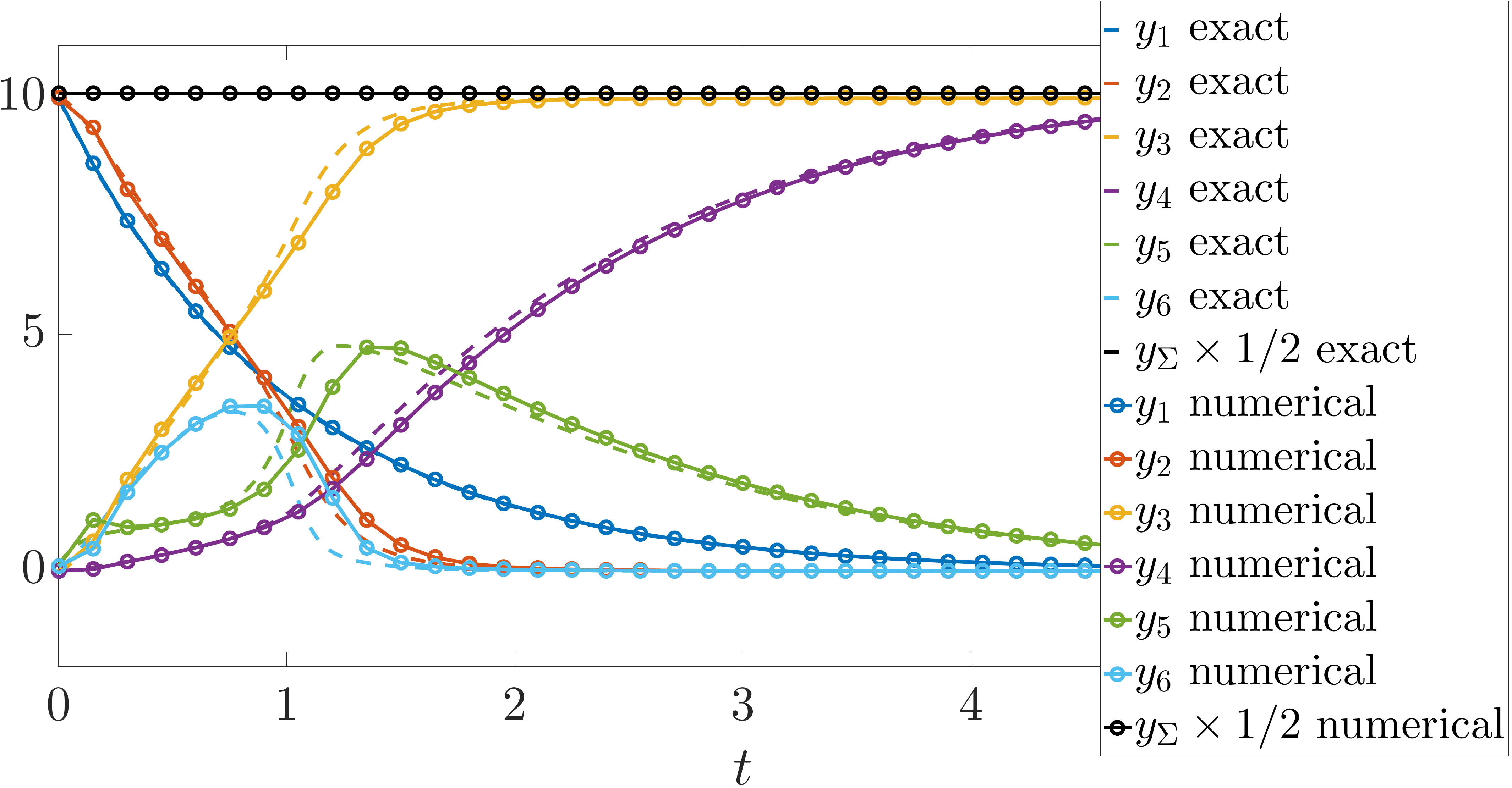}
 \caption{MPRK22($1/2$)}
 \end{subfigure}
 \hfill
\begin{subfigure}{.49\textwidth}
\centering
 \includegraphics[width=\textwidth]{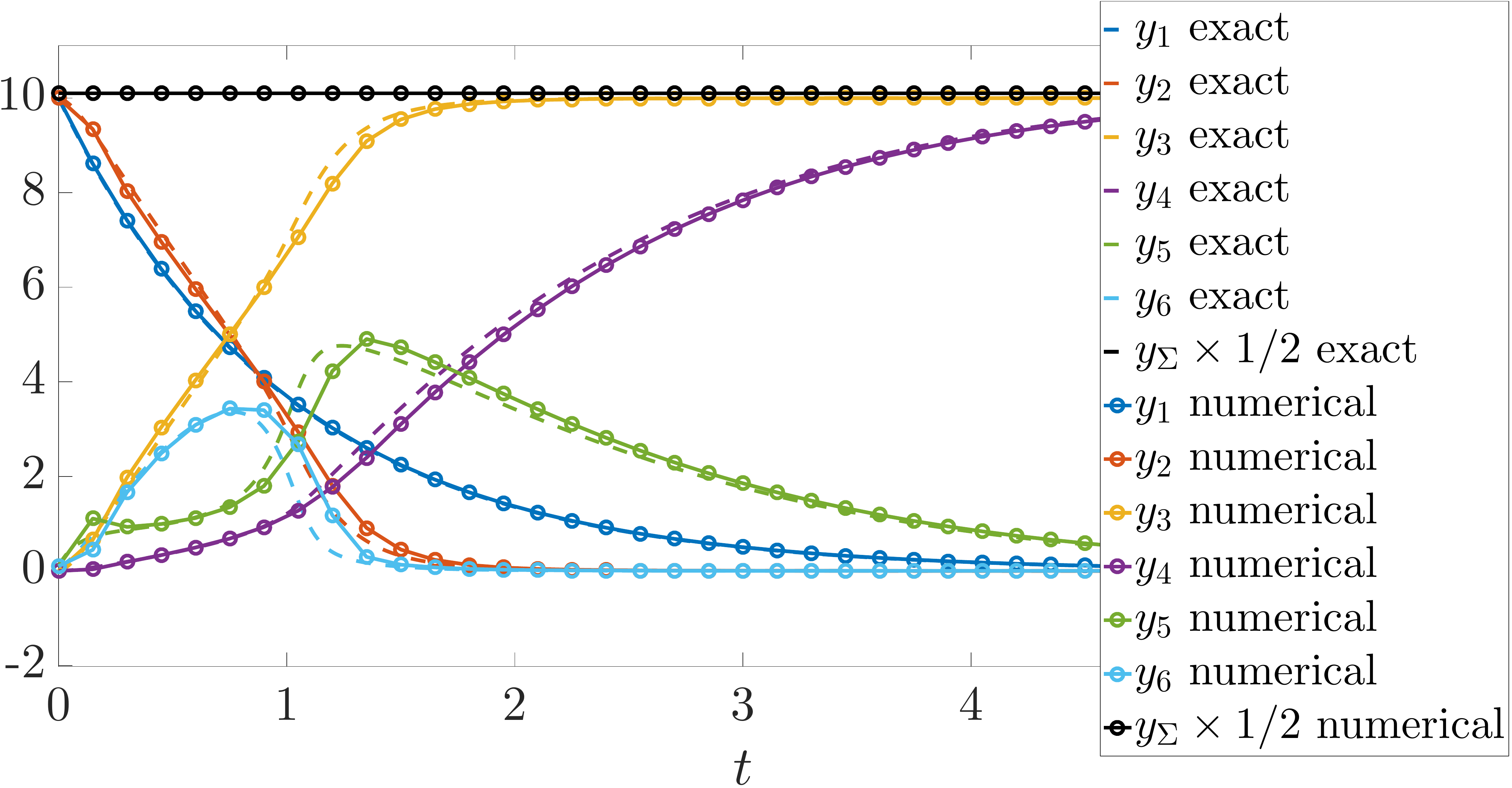}
 \caption{MPRK22ncs($1/2$)}
 \end{subfigure} 
 \par
\vspace{2\baselineskip} 
 \begin{subfigure}{.49\textwidth}
\centering
 \includegraphics[width=\textwidth]{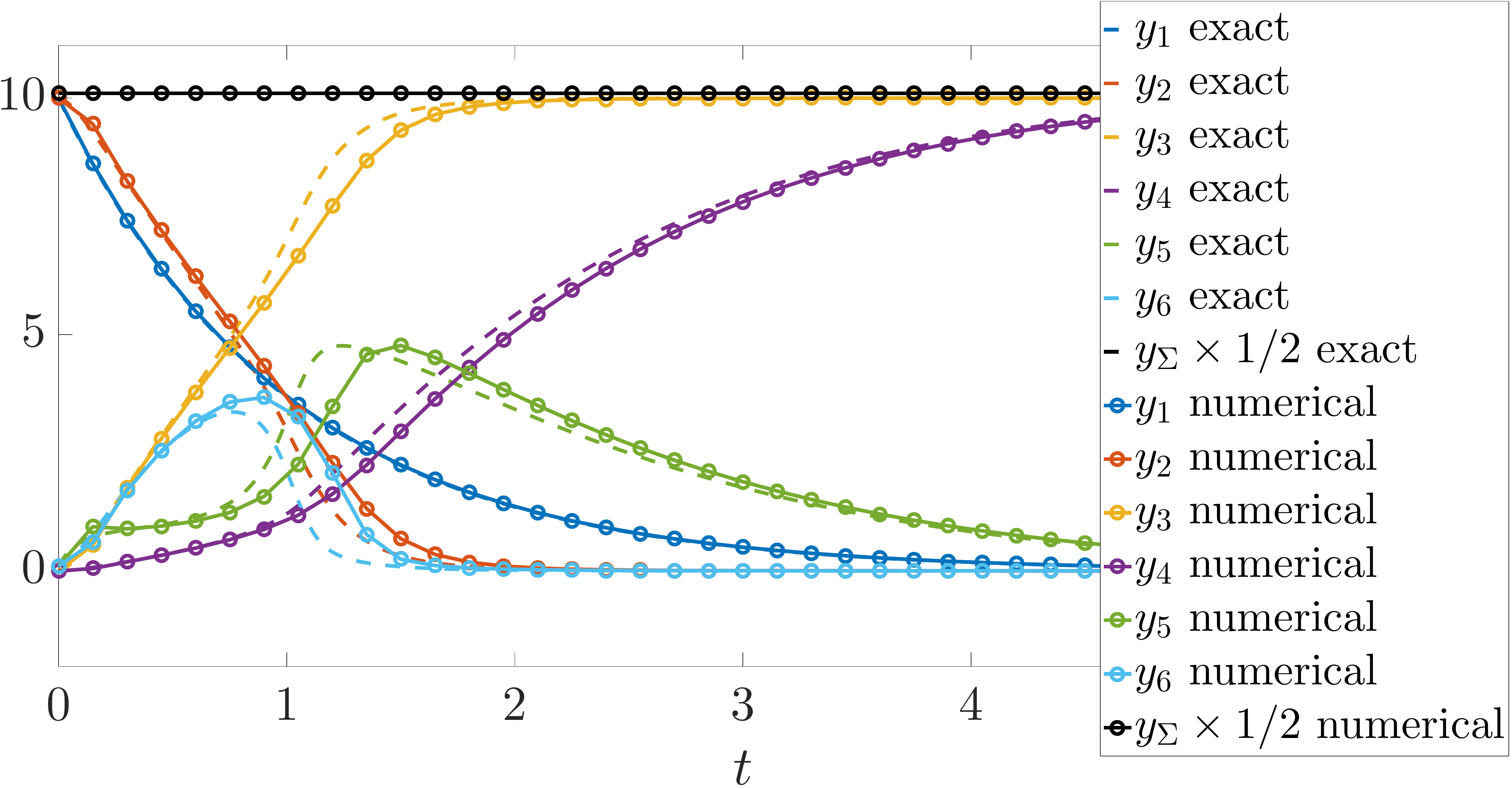}
 \caption{MPRK22($2/3$)}
 \end{subfigure}
\hfill
 \begin{subfigure}{.49\textwidth}
\centering
 \includegraphics[width=\textwidth]{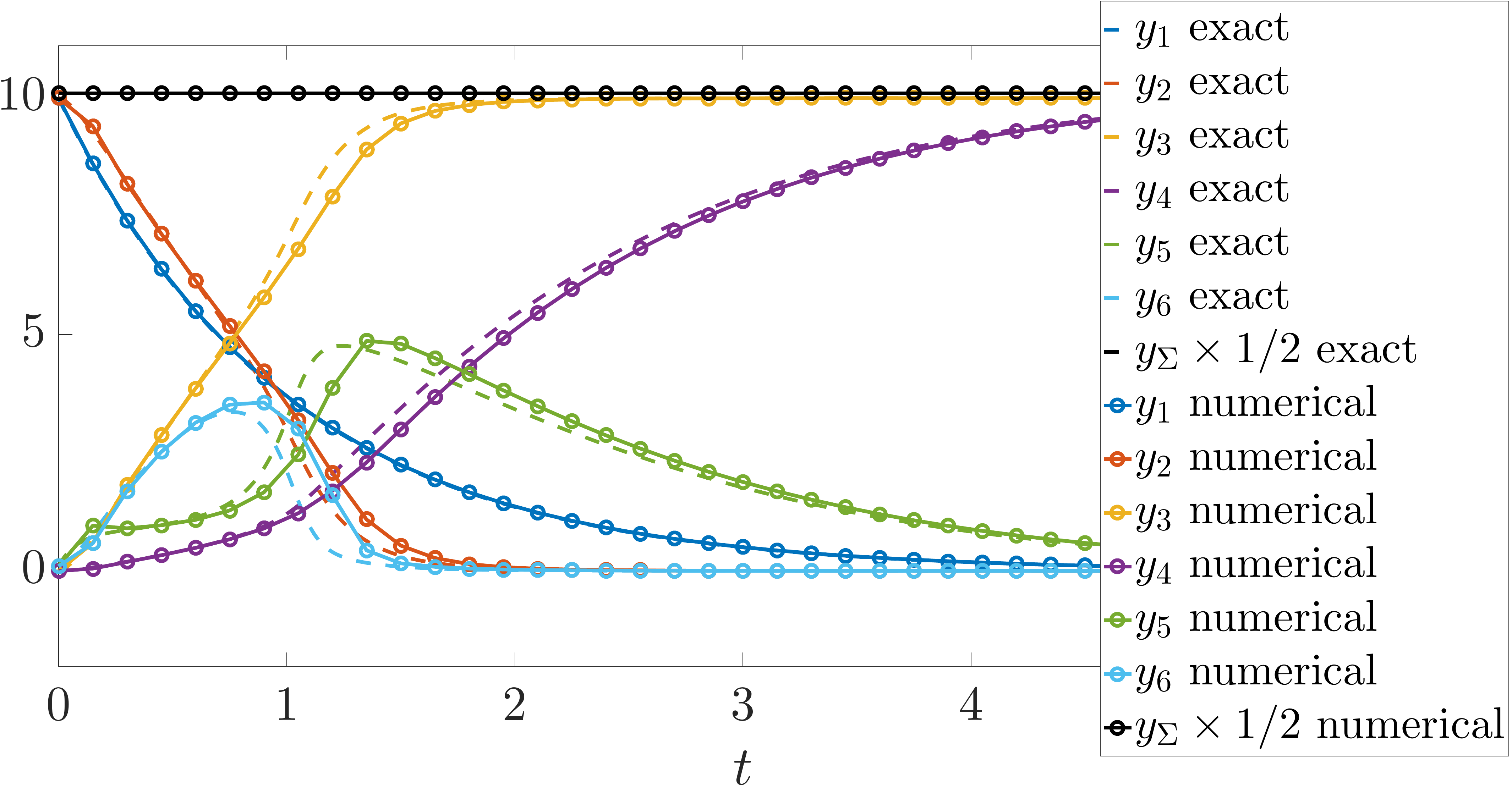}
 \caption{MPRK22ncs($2/3$)}
 \end{subfigure}
  \par
 \vspace{2\baselineskip}
  \begin{subfigure}{.49\textwidth}
\centering
 \includegraphics[width=\textwidth]{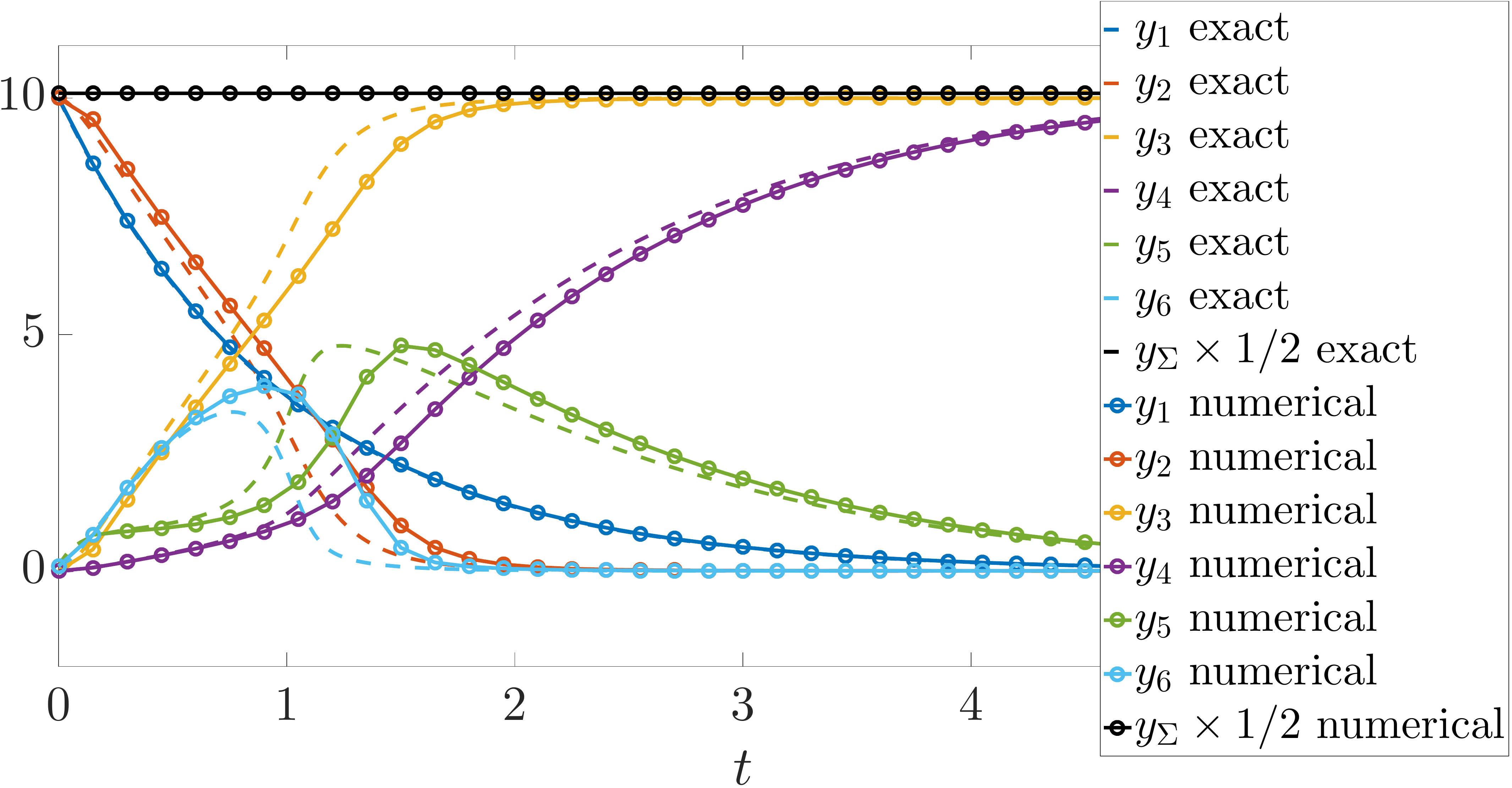}
 \caption{MPRK22(1)}
 \end{subfigure}
 \hfill
  \begin{subfigure}{.49\textwidth}
\centering
 \includegraphics[width=\textwidth]{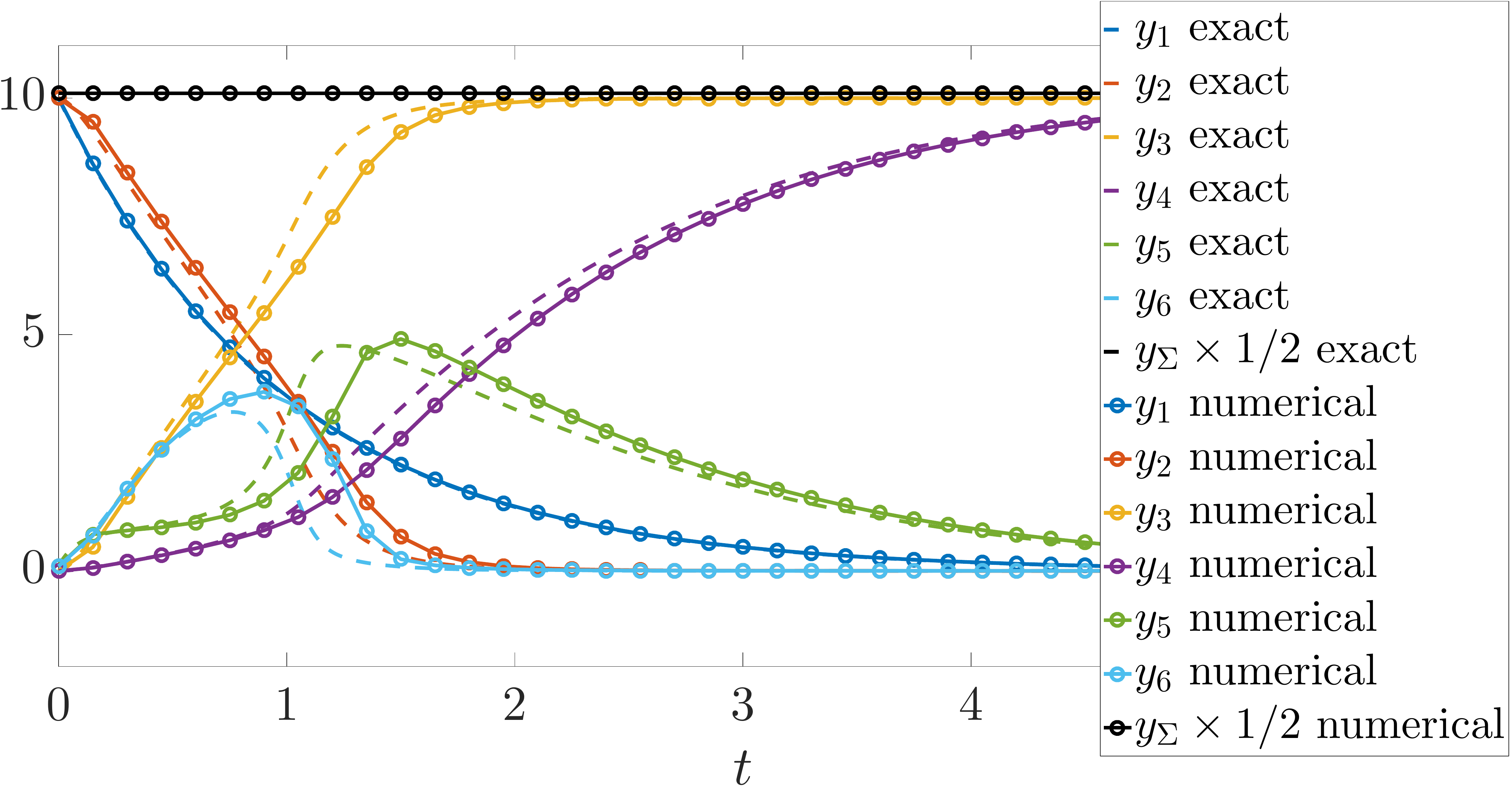}
 \caption{MPRK22ncs(1)}
 \end{subfigure}
 \caption{Numerical solutions of the Brusselator problem \eqref{eq:brusselator} for different MPRK22 and MPRK22ncs schemes. The term $y_\Sigma$, which appears in the legend, is defined as $y_\Sigma=y_1+\dots+y_6$.}
 \label{fig:brusselator}
\end{figure}
\section{Summary and Outlook}
In this paper we have introduced a general definition of modified Patankar-Runge-Kutta (MPRK) schemes, which includes the schemes originally introduced in \cite{BDM2003}.
We have shown that MPRK schemes are unconditionally positive and conservative by construction and introduced two novel families of second order MPRK schemes. 
The analysis concerning the order of MPRK schemes closes the gap to define both sufficient and necessary conditions with respect to the convergence order and yields a comprehensive investigation of first and second order schemes for the first time.

Numerical experiments confirmed the theoretical convergence order of these schemes and indicate that MPRK22($1/2$) is the preferable scheme in terms of truncation errors. 
They also demonstrated the capability of the MPRK22 schemes to integrate stiff PDS like the Robertson problem and revealed issues with oscillations of the MPRK22ncs schemes, when applied to the Robertson problem.

The numerical results motivate an analytical investigation of the truncation errors of MPRK22 schemes and a stability analysis of MPRK schemes in general.

Furthermore, the analysis carried out in this paper can be extended to schemes of order three and higher.
\newcommand{\etalchar}[1]{$^{#1}$}

\end{document}